\documentclass[11pt,reqno]{amsart}

\usepackage{amssymb,amsmath,amsthm,amsfonts}
\usepackage{bbm}
\usepackage{a4wide}
\usepackage{enumerate}
\usepackage{eucal}
 \usepackage[usenames,dvipsnames]{color}
\usepackage{bookmark}
\usepackage{hyperref}
\usepackage{mathrsfs}
\usepackage{recycle}
\usepackage{calligra}
\usepackage{wela}

\hypersetup{pdfstartview={FitH}}

\newtheorem{theorem}{Theorem}
\newtheorem{lemma}[theorem]{Lemma}
\newtheorem{corollary}[theorem]{Corollary}
\newtheorem{conjecture}[theorem]{Conjecture}
\newtheorem{proposition}[theorem]{Proposition}
\newtheorem{preremark}[theorem]{Remark}
\newenvironment{remark}{\begin{preremark}\rm}{\end{preremark}}
\newtheorem{prenotation}[theorem]{Notation}
\newenvironment{notation}{\begin{prenotation}\rm}{\end{prenotation}}

\numberwithin{equation}{section}
\numberwithin{theorem}{section}

\font\gotxi=eufm10 at 11pt
\font\posebni=msam10

\newcommand{\C}{\mathbb{C}}
\newcommand{\cR}{\mathcal R}
\newcommand{\cS}{\mathcal S}

\newcommand{\cU}{\mathcal U}
\newcommand{\cV}{\mathcal V}
\newcommand{\cW}{\mathcal W}
\newcommand{\pd}{\partial}
\newcommand{\e}{\varepsilon}
\newcommand{\leqsim}{\,\text{\posebni \char46}\,}
\newcommand{\geqsim}{\,\text{\posebni \char38}\,}

\newcommand{\nor}[1]{\| #1 \|}

\newcommand{\mn}[2]{\left\{ #1\, ;\, #2 \right\}}

\renewcommand{\leq}{\leqslant}
\renewcommand{\geq}{\geqslant}
\renewcommand{\searrow}{\downarrow}

\renewcommand\mod[1]{\left\vert{#1}\right\vert}

\newcommand\wrt{\,d{}}
\newcommand\norm[2]{{\big\Vert{#1}\big\Vert_{#2}}}

\title[Complex Riesz transform]{Sharp $L^p$ estimates of powers of the complex Riesz transform}

\author[A. Carbonaro]{Andrea Carbonaro}
\address{Andrea Carbonaro, Universit\`{a} degli Studi di Genova, Dipartimento di Matematica, Via Dodecaneso 35, 16146 Genova, Italy}
\email{carbonaro@dima.unige.it}

\author[O. Dragi\v{c}evi\'{c}]{Oliver Dragi\v{c}evi\'{c}}
\address{Oliver Dragi\v{c}evi\'{c}, Department of Mathematics, Faculty of Mathematics and Physics, University of Ljubljana, and Institute for Mathematics, Physics and Mechanics, Jadranska 21, SI-1000 Ljubljana, Slovenia}
\email{oliver.dragicevic@fmf.uni-lj.si}

\author[V. Kova\v{c}]{Vjekoslav Kova\v{c}}
\address{Vjekoslav Kova\v{c}, Department of Mathematics, Faculty of Science, University of Zagreb, Bijeni\v{c}ka cesta 30, 10000 Zagreb, Croatia}
\email{vjekovac@math.hr}

\begin{document}

\begin{abstract}
Let $R_{1,2}$ be scalar Riesz transforms on $\mathbb{R}^2$. We prove that the $L^p$ norms of $k$-th powers of the operator $R_2+iR_1$ behave exactly as $|k|^{1-2/p^\ast}(p^\ast-1)$, uniformly in $k\in\mathbb{Z}\backslash\{0\}$ and $1<p<\infty$, where $p^\ast$ is the bigger number between $p$ and its conjugate exponent.
This gives a complete asymptotic answer to a question suggested by Iwaniec and Martin in 1996. The main novelty are the lower estimates, of which we give three different proofs. We also conjecture the exact value of $\|(R_2+iR_1)^k\|_p$. Furthermore, we establish the sharp behaviour of weak $(1,1)$ constants of $(R_2+iR_1)^k$ and an $L^\infty$ to $BMO$ estimate that is sharp up to a logarithmic factor.
\end{abstract}

\maketitle

%%%%%

\section{Introduction and the statement of the main results}

Let $R_1$ and $R_2$ be the classical {\it Riesz transforms} on $\mathbb{R}^2\equiv\mathbb{C}$, see Grafakos \cite[Section 5.1]{G14} for definitions and main properties. We will be concerned with $\nor{R^k}_p$, by which we denote the $L^p$ norms ($1<p<\infty$) of integer powers of the operator
\begin{equation*}
R=R_2+iR_1.
\end{equation*}
Possibly under different choices of
normalization,
$R$ is sometimes called the {\it complex Hilbert transform} \cite{IM96} or {\it complex Riesz transform} \cite[Section 4.2]{AIM09}.
The operator $R$ is the Fourier multiplier associated with the symbol
\begin{equation}
\label{eq: Four symbol}
\zeta\mapsto\frac{\overline{\zeta}}{|\zeta|}; \quad \zeta\in\mathbb{C}\setminus\{0\}.
\end{equation}
This is also a singular integral operator of convolution type.
By \cite[Theorem 4.2.1]{AIM09}, the (principal-value) convolution kernel of the $k$-th power $R^k$ of $R$ is for $k\in\mathbb{Z}$ given by
\begin{equation}
\label{eq: Opel Kadett}
\Omega_k(z):=\frac{i^{|k|}|k|}{2\pi }\ \frac{(z/|z|)^{-k}}{|z|^2}\,.
\end{equation}
That is, for admissible functions $f:\mathbb{C}\rightarrow\mathbb{C}$ we have
\begin{equation}
\label{Uesugi Kenshin}
(R^{k}f)(z)=
\frac{i^{|k|}|k|}{2\pi }
\lim_{\delta\searrow0}\int_{\{|w|>\delta\}}
\frac{(w/|w|)^{-k}}{|w|^2}f(z-w)\,dA(w).
\end{equation}
Here $dA$ stands for the usual planar
(Lebesgue) measure on $\mathbb{C}$.

Of particular interest is
the square of $R$; the operator $R^2$ is
known as the {\it Ahlfors--Beurling} (or {\it Beurling--Ahlfors}, or just {\it Beurling}, sometimes even {\it Hilbert} \cite{LV65}) {\it transform}.

\medskip
The importance of $\nor{R^k}_p$
 for arbitrary $k\in\mathbb{Z}$
was first understood in the influential 1996 paper \cite{IM96} by Iwaniec and Martin.
Throughout their article they extensively work with $R^k$ and
most of their estimates are expressed
in terms of $\nor{R^k}_p$. However,
on p.~27--28 they write:

\medskip
{\it
``We should point out however that the $p$-norms [\dots] of the $m^{th}$ iterate of the complex Hilbert transform are as yet unknown.''}

\medskip
For $a_1,a_2>0$ we write $a_1\geqsim a_2$ if there is a constant $C>0$ such that $a_1\geqslant C a_2$. Similarly we define $a_1\leqsim a_2$. If both $a_1\geqsim a_2$ and $a_1\leqsim a_2$, then we write $a_1\sim a_2$.
Furthermore, given $p>1$, the letter $q$ will from now stand for its conjugate exponent $q=p/(p-1)$ unless specified otherwise. Moreover, we will use the notation
$$
p^*=\max\{p,q\}.
$$

In \cite{DPV06} and \cite{D11-2} it was proven that
\begin{equation}
\label{13}
\nor{R^k}_p\sim k^{1-2/p^*}(p^*-1)
\hskip 30pt
\text{for }even\ k\in\mathbb{N}
\end{equation}
in the sense of bilateral estimates with constants independent of $k\in\mathbb{N}$ and $p>1$.
Unaware
of \cite{DPV06},
Astala, Iwaniec and Martin proved a bit later
that $\nor{R^k}_p$ grows at most quadratically with respect to $k$, see their monograph \cite[Corollary 4.5.1]{AIM09}.
They also pointed out that the {\it exact} growth rate was unknown \cite[p. 124]{AIM09}:

\medskip
{\it ``In the sequel it will be important to establish that the $L^p(\mathbb{C})$ norms of the
$k^{\text{th}}$ iterate of the complex Riesz transform do not grow as a power but rather
that the growth is a polynomial in $k$. It is not at all clear what the precise
asymptotics here are.''}

\medskip
In this paper we prove that \eqref{13} also holds for {\it odd} powers $k$, which thus completely answers
the above questions in the asymptotic sense:
\begin{theorem}
\label{t: main}
We have
\begin{equation*}
\nor{R^k}_p\sim k^{1-2/p^*}(p^*-1)
\qquad
\text{for all } k\in\mathbb{N}
\text{ and }p\in(1,\infty).
\end{equation*}
The equivalence constants are absolute, that is, independent of $p$ or $k$.
\end{theorem}

We complement this
with the sharp weak 1-1 estimate of $R^k$. It improves \cite[(4.82)]{AIM09},
where a quadratic estimate in terms of $k$ was established.
\begin{theorem}
\label{v-o-d-a}
We have
\[ \nor{R^k}_{L^1\rightarrow L^{1,\infty}}\,\sim\,k
\qquad
\text{for all } k\in\mathbb{N}. \]
\end{theorem}

One might also ask about
rounding off
Theorem \ref{t: main} with estimates on $L^\infty$. That is, we wonder about the sharp behaviour of norms of $R^k\colon L^\infty\rightarrow BMO$ with respect to $k$. When $k=2$, Iwaniec \cite{I86} computed the exact norm of $R^2\colon L^\infty\rightarrow BMO_2$ and it equals $3$.
His proof heavily uses the fact that the integral kernel of the Ahlfors--Beurling operator, $\Omega_2$, is holomorphic on $\mathbb{C}\backslash\{0\}$. Since $\Omega_k$ is holomorphic {\it precisely} for $k=2$, extending
\cite{I86}
to arbitrary $R^k$ would likely require a
different argument. However we are still able to deduce a result that is sharp up to a logarithmic factor:
\begin{theorem}
We have
\label{thm: LinftyBMObound}
\[ k \,\lesssim\, \nor{R^k}_{L^\infty\rightarrow BMO} \,\lesssim\, k \log(k+1)
\qquad
\text{for all } k\in\mathbb{N}. \]
\end{theorem}

\subsection{Main focus
-- lower $L^p$ estimates for odd $k$}

As said above,
our primary goal is to
show that \eqref{13} also holds for {\it odd} $k\in\mathbb{N}$ (Theorem \ref{t: main}).
The easy part are the upper estimates:
indeed, by recycling
the method applied in \cite{DPV06} for the proof of the upper estimate in \eqref{13}, we
see
that
\begin{equation}
\label{31}
\nor{R^k}_p\leqsim k^{1-2/p^*}(p^*-1)
\end{equation}
 in fact holds for {\it all} positive integers $k$, not just the even ones
 treated in \cite{DPV06}. See \cite[Theorem 2.38]{D20}; for the reader's convenience, the proof there is reproduced in Section~\ref{Appendix}.

Taking \eqref{31} for granted, by \eqref{13}
the ``missing'' part---and the main effort in
this article---is now the following lower estimate.
\begin{theorem}
\label{t: Lowest}
We have
\begin{equation}
\label{Lowest}
\nor{R^k}_p\,\geqsim\, k^{1-2/p^*}(p^*-1)
\hskip 30pt
\text{for odd }k\in\mathbb{N}
\text{ and }p\in(1,\infty).
\end{equation}
The equivalence constants are absolute, that is, independent of $p$ or $k$.
\end{theorem}

We will present {\it three} different proofs of
Theorem~\ref{t: Lowest},
each of them having certain conceptual or aesthetic advantages, which is why we decided to include them all in this paper.
As a by-product of two of these proofs we establish a lower bound with an explicit constant (see Theorem~\ref{t: LowerExact}), which leads us to conjecture the exact values of $\nor{R^k}_p$ as follows:
\begin{conjecture}
\label{Konj}
For all $k\in\mathbb{N}$ and $p\geq2$ we have
$$
\nor{R^k}_p
=
\frac{\Gamma(1/p)\Gamma(1/q+k/2)}{\Gamma(1/q)\Gamma(1/p+k/2)}
$$
where $1/p+1/q=1$.
When $1<p<2$, the identity above should still hold, but with the roles of $p,q$ reversed, in accordance with Section~\ref{Clear}.

\end{conjecture}
We arrived at this conjecture by first proving one direction $(\geq)$ of the above identity, see
\cite[Theorem 1.1]{D11-2} for the even case and Section~\ref{s: Vjeko's} for the odd case, while Section~\ref{sec: Andrea proof} gives a proof that unifies both cases.
Conjecture~\ref{Konj} was in the case of even $k$
formulated in \cite[p.508]{DPV06} and \cite[Conjecture 1.2]{D11-2}, while the case $k=2$ has been known for about 40 years as the {\it Iwaniec conjecture} and simply reads
$\nor{R^2}_p=p^\ast-1$, see
\cite[Conjecture 1]{I82}.
It remains open at the moment of this writing.

\medskip
The
difficulties
we are facing while
proving
Theorem~\ref{t: Lowest}
are due to several reasons,
on top of
those explained in Section~\ref{s: Odd case} below:
\begin{enumerate}[(i)]
\item
we are dealing with oscillatory integrals,
\item
we are trying to estimate such integrals {\it from below},
\item
we want sharp estimates, and
\item
the sharpness should be attained simultaneously in {\it two} parameters ($k,p$).
\end{enumerate}

\subsection{Why is the odd case different?}
\label{s: Odd case}
Unlike when addressing the {\it upper} $L^p$ estimates \eqref{31},
when pursuing
the {\it lower} estimates for $R^k$ with {\it odd} $k$, we
cannot simply mimic the proofs from the even case; the latter were hinted at in \cite{DPV06} and rigorously carried out in \cite{D11-2}. For while $R^{2n}$ has a ``nice'' description in the sense of mapping $\pd_{\bar z}^n f\mapsto\pd_{z}^n f$ \cite[Lemmas V.2,3]{Ahlfors}, \cite[Hilfssatz 7.2]{LV65}, \cite[(4.18)]{AIM09},
which was used in \cite{DPV06} and \cite{D11-2},
the only analogous representation that $R$ admits is
\begin{equation}
\label{eq: Delta repres}
(-\Delta)^{1/2}f\mapsto -2i\pd_{z}f
\end{equation}
(which quickly follows by applying the Fourier transform).
See \cite[Example 3.7.5]{D} or Section~\ref{sec: Andrea proof} for details regarding $(-\Delta)^{1/2}$.
Consequently, $(-\Delta)^{1/2}$ intervenes in representations of all {\it odd} powers of $R$. However $(-\Delta)^{1/2}$ is not a
local operator, which makes dealing with it
considerably more difficult.

\medskip
Still, one can say something about the lower estimates of $\nor{R^k}_p$ for odd $k=2l-1$ based on the even case. Namely, we obviously have
$\nor{R^{4l-2}}_p \leq\nor{R^{2l-1}}_p^2$ and $\nor{R^{2l}}_p\leq\nor{R^{2l-1}}_p\nor{R}_p$
which, together with \eqref{13}, \eqref{31} and the estimate $\nor{R}_p\leqsim p^*-1$ \cite{IM96}, gives
\medskip
\noindent
\begin{equation}
\label{Last Month of the Year}
\max\left\{
l^{1-2/p^*},l^{1/2-1/p^*}\sqrt{p^*-1}
\right\}
\leqsim
\nor{R^{2l-1}}_p
\leqsim
l^{1-2/p^*}(p^*-1),
\end{equation}
for all $l\in\mathbb{N}$ and $p>1$, with the implicit constants independent of $l,p$.

\begin{remark}
The behaviour of $L^p$ estimates of powers of the Ahfors--Beurling operator $R^2$ has also been studied on {\it weighted} $L^p$ spaces. The first paper in this direction was \cite{D11-2}. The estimates obtained there were significantly improved by Hyt\"onen \cite{Hy11, Hy14} and Hyt\"onen, Roncal and Tapiola \cite{HRT17}. See \cite[Section 2.8]{D20} for a summary of these results.
\end{remark}

\subsection{A few reductions}

\subsubsection{Invariance of norms of powers of $R$}
\label{Clear}
Recall that the powers $R^k$ were also defined for {\it negative} integers $k$. We have
that $R^k$ is the inverse of $R^{-k}$ \cite[p.~102]{AIM09}
on $L^2(\mathbb{C})$.
From \eqref{eq: Opel Kadett} we get
$\Omega_{-k}(z)=\Omega_k(\bar z)=\overline{\Omega_k(z)}$,
which quickly implies $R^{-k}f=\overline{R^k\bar f}$.
Consequently,
$
\nor{R^{-k}}_p=\nor{R^{k}}_p.
$
Since $R^{-k}$ is the $L^2(\mathbb{C})$-adjoint of $R^k$ \cite[p.~102]{AIM09}, it follows that $\nor{R^k}_p=\nor{R^k}_{q}$.
So
$\nor{R^k}_p$
depend on $|k|$ and $p^*$ only. Therefore we may assume that $k\in\mathbb{N}$ and $p\geq2$.

\subsubsection{Case $k=1$}
We separately consider \eqref{Lowest} for $k=1$, which now reads
\begin{equation}
\label{Leadbelly}
\nor{R}_p\,\geqsim\, p
\hskip 30pt
\text{for } p\in(2,\infty).
\end{equation}

Let $f\in L^p(\mathbb{C})$ be a real-valued function. Since the scalar Riesz transforms $R_1,R_2$ map
real-valued functions back into real-valued functions, we have
$
|Rf|=|R_2f+iR_1f|\geq |R_1f|,
$
therefore
$
\nor{Rf}_p \geq \nor{R_1f}_p.
$
Owing to a theorem attributed to Marcinkiewicz and Zygmund \cite[Proposition 3.1]{IM96}, the norm of $R_1$ on the space $L^p$ (understood, as usual, as a space of complex-valued functions) is equal to the $L^p$ norm of $R_1$ over real-valued functions only. Therefore we may find a nonzero real-valued $f\in L^p(\mathbb{C})$ such that $\nor{R_1f}_p\geq \nor{R_1}_p\nor{f}_p/2$.
Now \cite[Theorem 1.1]{IM96} gives $\nor{R_1}_p=\cot(\pi/(2p))\geqsim p$ for $p\geq2$, which proves \eqref{Leadbelly}.

For a more general treatment of norms of operators in relation to complexifications see the monograph by Hyt\"onen et al. \cite[Section 2.1]{HvNVW16}.

\subsubsection{Case $p\approx 2$}
When $p\approx 2$, the estimate $\nor{R^k}_p\geqsim k^{1-2/p^*}$, provided by \eqref{Last Month of the Year}, is
equivalent to the desired estimate
\eqref{Lowest}, because $2l-1\sim l$ and for $p\approx 2$ we have $p^*-1\sim 1$.
This settles Theorem \ref{t: Lowest} for $p\approx 2$.

\medskip
These observations together imply that, unless stated otherwise, we may
\begin{equation}
\label{kp3}
\text{
\it{from now on assume that $k,p\geq3$, where $k$ is an odd integer.}
}
\end{equation}

\subsection{The $L^p$ spectrum of $R$}
Finally, we present a simple corollary of Theorem~\ref{t: main}, modelled after \cite[Corollary 1 and Theorem 2]{D11}. Let $\sigma(R)$ denote the {\it spectrum} of $R$ and $\sigma_c(R)$ its {\it continuous spectrum}.
\begin{corollary}
Fix $p>1$ and consider $R$ as a bounded operator on $L^p$. Then
$$
\sigma(R)
=
\sigma_c(R)
=
\mn{\zeta\in\mathbb{C}}{|\zeta|=1}.
$$
\end{corollary}
The result follows by modifying the proofs from \cite{D11} where the same was proved with $R^2$ in place of $R$, and using the trivial inclusion Im$(R^2-I)\subset$ Im$(R-I)$.

\subsection{Paper outline}
As mentioned above, we present three proofs of the paper's main result, Theorem~\ref{t: Lowest}.
All three involve certain sequences which essentially approximate the same extremal function and which will be motivated in Section~\ref{sec: approx seq}.
However, the proofs feature different approaches, of course.

The first proof, presented in Sections~\ref{s: baba brez maske}--\ref{Zemi ogin}, is an ``integral kernel'' proof.
It relies on the very explicit approximating sequence from \eqref{eq: approx in Oliver} below and the integral representation \eqref{Uesugi Kenshin}.
Overall, the proof is completely elementary and entails considerations exclusively on the spatial side (as opposed to the Fourier domain).

The second proof, presented in Section~\ref{s: Vjeko's}, could be called the ``Gaussian--Fourier'' proof and relies on the fact that $R$ is the Fourier multiplier with symbol \eqref{eq: Four symbol}.
An advantage of this approach is that it is significantly shorter and it also yields an explicit constant in the lower bound for $\|R^k\|_p$.
Now the main idea is to work on the Fourier side and construct the approximating functions as superpositions of dilated Gaussians. This makes them less explicit, but their inverse Fourier transforms can be computed and estimated in the $L^p$ norm.

The third proof, presented in Section~\ref{sec: Andrea proof},
is a ``Riesz potential'' proof and it relies on the representation \eqref{eq: Delta repres} of $R$.
It shares features with both preceding proofs and it was, in fact, heavily motivated by them. It uses an extremizing sequence similar to the one used in the first proof, but it interprets it rather as smooth truncations of the Riesz potentials applied to the Dirac distribution, on which $R^k$ acts quite elegantly. Similarities with the second proof become visible after one recalls that Riesz and Bessel potentials can be conveniently decomposed into superpositions of Gaussians: a trick which can be traced back at least to Stein \cite[Chapter~V, \S 3.1]{S70}.
This proof gives the same explicit constant as the second one and it works equally well for odd and even $k\in\mathbb{N}$.
Moreover, on the formal level it explains the constant from Theorem~\ref{t: LowerExact} (or Conjecture~\ref{Konj}) in a rather straightforward manner.

Finally, Section~\ref{Appendix} completes the proof of Theorem~\ref{t: main} by elaborating on the upper bound. Short proofs of Theorems~\ref{v-o-d-a} and \ref{thm: LinftyBMObound} are also presented in that section.

\section{Construction of the approximating sequence}
\label{sec: approx seq}

Let us phrase the very basic idea behind the proofs.

\subsection{Model case: even powers \cite{D11-2}}
When proving $\nor{R^{2n}}_p\geqsim n^{1-2/p^*}(p^*-1)$ in \linebreak\cite{D11-2}, the second-named author of the present paper constructed an adequate sequence $f_\varepsilon=f_{n,\varepsilon}$ for which he then calculated
$$
\lim_{\varepsilon\searrow0}\frac{\nor{R^{2n} (\pd_{\bar z}^nf_\varepsilon)}_p }{\hskip 18pt \nor{\pd_{\bar z}^nf_\varepsilon}_p}
$$
and proved that the outcome was comparable to $n^{1-2/p^*}(p^*-1)$.
Here
$\pd_{\bar z}^nf_\varepsilon$ was of the form
$$
C(n,\varepsilon)\left(\frac{z}{|z|}\right)^{2n}|z|^{-2(1/p-\varepsilon)}\chi_ D (z)+(E.T.)\chi_{ D ^c}(z),
$$
with
$D$ denoting the open unit disc in $\mathbb{C}$ and
``$(E.T.)$'' standing for ``error term''.
Its construction was driven by
the requirement that $f_{n,\varepsilon}$ belong to the Sobolev space $W^{n,p}(\mathbb{C})$.

\subsection{Basic setup}
Therefore in order to reach $R^k$ for arbitrary integer $k$, a reasonable starting point seems to be the family of functions $F_\varepsilon=F_{k,p,\varepsilon}$, defined by
\begin{equation*}
F_\varepsilon(z):=\left(\frac{z}{|z|}\right)^{k}|z|^{-2/p+2\varepsilon}\chi_ D (z).
\end{equation*}
In order to prove Theorem~\ref{t: Lowest} we need
to understand if indeed
\begin{equation}
\label{Eminem Mariah Carey Diss}
\liminf_{\varepsilon\searrow0}\frac{\nor{R^kF_\varepsilon}_p }{\nor{F_\varepsilon}_p}\,\geqsim\, k^{1-2/p}(p-1)
\end{equation}
uniformly in $k\in\mathbb{N}$ and $p>2$.
Clearly, $\nor{F_\varepsilon}_p\sim \varepsilon^{-1/p}$.

\subsection{Switching to upper estimates}
The main task is to estimate the norm $\nor{R^kF_\varepsilon}_p$ from {\it below}.
However, estimating integrals from below, especially (oscillatory) singular integrals, seems difficult. Yet in our case we are in a position to formulate the problem in terms of equivalent {\it upper} estimates.
Indeed, by
writing ${\text{\gotxi F}}_\varepsilon={\text{\gotxi F}}_{k,\varepsilon}:=R^kF_\varepsilon$ we can express \eqref{Eminem Mariah Carey Diss} as
\begin{equation}
\label{Ciklovir}
\limsup_{\varepsilon\searrow0}\frac{\nor{R^{-k}\text{\gotxi F}_\varepsilon}_p}{\nor{\text{\gotxi F}_\varepsilon}_p}
\leqsim\frac{1}{k^{1-2/p}(p-1)}
\,.
\end{equation}
The problem is that we do not know an explicit ``algebraic'' formula for $\text{\gotxi F}_\varepsilon$.
Again we resort to \cite{D11-2} for hints. Namely,
\begin{itemize}
\item
for even $k$, we ``almost'' have $(R^kF_\varepsilon)\big|_ D \equiv |z|^{-2/p+2\varepsilon}$ (cf. \cite[p. 466]{D11-2}),
\item
experience from \cite{D11-2} also suggests that what matters is $F_\varepsilon\big|_ D $ and $(R^kF_\varepsilon)\big|_ D $.
\end{itemize}
Thus the idea is to apply $R^{-k}$ not to ${\text{\gotxi F}}_\varepsilon$, but instead to the
 function $\Phi_\varepsilon$, defined {\it explicitly} by
\begin{equation}\label{eq: approx in Oliver}
\Phi_\varepsilon(z):=|z|^{-2/p+2\varepsilon}\chi_{ D }(z).
\end{equation}
We readily see that $\nor{\Phi_\varepsilon}_p=\nor{F_\varepsilon}_p\sim\varepsilon^{-1/p}$.

Observe that $\Phi_\varepsilon$ does {\it not} depend on $k$, unlike $F_\varepsilon$. Still, we will show that $(\Phi_\varepsilon)_{\varepsilon>0}$ constitutes an extremal sequence that establishes the sharp lower bound for $\nor{R^k}_p$ for {\it every} odd $k\in\mathbb{N}$.

\subsection{Useful variants}
What we showed is not the only reasonable way to approximate $|z|^{-2/p}\chi_ D (z)$. One variant, used in Section~\ref{s: Vjeko's}, starts with the observation that the Fourier transform of the tempered distribution $|z|^{-2/p}$ is a constant multiple of $|z|^{-2+2/p}$. These homogeneous functions can be conveniently realized as superpositions of dilated Gaussians and, this time, it is more convenient to localize them in the frequency domain.
Another variant, used in Section~\ref{sec: Andrea proof}, uses an approximating sequence closer to \eqref{eq: approx in Oliver}, but recognizes it rather as truncations of the Riesz potentials $(-\Delta)^{\alpha}\delta_0$.
However, each of the three variants keep either $|z|^{-2/p}$ or $|z|^{-2/p}\chi_ D (z)$ in mind as an ``ideal''---albeit not legitimate---extremizing function.

\section{The first proof of Theorem~\ref{t: Lowest}: Explication of $R^{-k}\Phi_\varepsilon$}
\label{s: baba brez maske}

By \cite[Corollary 4.9]{D01}, the identity \eqref{Uesugi Kenshin} is valid with $f=\Phi_\varepsilon$ almost everywhere on $\mathbb{C}$.
Hence
 we immediately obtain
$(R^{-k}\Phi_\varepsilon)(e^{i\psi}z)=e^{ik\psi}(R^{-k}\Phi_\varepsilon)(z)$.
Thus it suffices to study $(R^{-k}\Phi_\varepsilon)(u)$ for $u>0$.
Explicitly,
using the polar coordinates gives
\begin{equation*}
\nor{R^{-k}\Phi_\varepsilon}_p^p=2\pi\int_0^\infty|(R^{-k}\Phi_\varepsilon)(u)|^pu\,du\,.
\end{equation*}
Let us denote by $\cR_k$ the operator $R^k$ without the normalization $i^{|k|}|k|/(2\pi)$ from \eqref{eq: Opel Kadett}. The same for $\cR_{-k}$ and $R^{-k}$, of course. Then \eqref{Uesugi Kenshin} turns into
\begin{equation}
\label{Fade away}
(\cR_{-k}\Phi_\varepsilon)(z)=
\lim_{\delta\searrow0}\int_{\{|w|>\delta\}\cap\{|w-z|<1\}}
\frac{(w/|w|)^{k}}{|w|^2}|z-w|^{-2/p+2\varepsilon}\,dA(w).
\end{equation}
Hence the targeted inequality \eqref{Ciklovir} now, after
switching from ${\text{\gotxi F}}_\varepsilon$ to $\Phi_\varepsilon$, reads
\begin{equation}
\label{Kronos}
\limsup_{\varepsilon\searrow0}
\left(
\varepsilon\int_0^\infty|(\cR_{-k}\Phi_\varepsilon)(u)|^pu\,du
\right)^{1/p}
\leqsim\frac{1}{k^{2(1-1/p)}(p-1)}.
\end{equation}

In
the continuation
we further analyze $(\cR_{-k}\Phi_\varepsilon)(u)$, with the aim of obtaining more revealing expressions of this term. Our plan is to separately treat the cases of $0<u<1$ and $u>1$.
We will use the following notation: for $0<\e<1/p$ let
\begin{equation}
\label{19vodki}
\sigma=\sigma(p,\varepsilon):=1/p-\varepsilon.
\end{equation}
For $a>0$, $b\in(0,1)$ and $c=\sqrt{a^2+b^2}$ define $H=H_{\sigma}:(0,\infty)\times(0,1)\rightarrow\mathbb{R}$ as
\begin{equation}
\label{Scherbakov Symphony no.5}
\aligned
H({a},{b})
&=
\int_0^{c-b}
\frac{(1-2{b}s+s^2)^{-\sigma}-(1+2{b}s+s^2)^{-\sigma}}{s}
\,ds+
\int_{c-b}^{c+{b}}
\frac{(1-2{b}s+s^2)^{-\sigma}}{s}\,ds
\\
&=:H^0+H^1.
\endaligned
\end{equation}
In our applications we will typically use
\begin{equation*}
a=\sqrt{1/u^2-1},
\hskip 25pt
b=\cos\varphi,
\hskip 25pt
c=\sqrt{a^2+b^2}=\sqrt{1/u^2-\sin^2\varphi}.
\end{equation*}
This use will arise (and in case of $a$ only makes sense) when $0<u<1$.

\begin{proposition}
\label{Markovina Velikic}
Fix $p\in(1,\infty)$, $\varepsilon>0$ and an odd $k\in\mathbb{N}$.
Write $\sigma=1/p-\varepsilon$ as in \eqref{19vodki} and
let $H_\sigma$ be given by
\eqref{Scherbakov Symphony no.5}.
For any $u\in(0,1)\cup(1,\infty)$ we have
\begin{equation}
\label{Utushka lugovaya}
(\cR_{-k}\Phi_\varepsilon)(u)=2u^{-2\sigma}G_\sigma(u),
\end{equation}
where $G_\sigma=G_{k,\sigma}:(0,\infty)\rightarrow\mathbb{R}$ is defined as follows:
\vskip 10pt
\ \hskip 5pt
$
{\displaystyle G_\sigma(u)
=
\left\{
\begin{array}{c}
\ \\
\ \\
\ \\
\ \\
\ \\
\end{array}
\right.
}
$
\vskip - 73pt
\hskip 30pt
\begin{align}
\label{Kalinnikov Symphony no.1}
&\ \hskip 50pt {\displaystyle
 \int_0^{\pi/2}
H_\sigma\left(\sqrt{1/u^2-1},\cos\varphi\right)
\cos k\varphi
\,d\varphi
}
\ &
 & \text{ if }0<u<1,
\\
\label{Braun HC5030}
&\ \hskip 50pt {\displaystyle
\int_0^{\pi/2}
\int_0^{1/u^2}
\cos(k\arcsin(\sqrt{s}\cos\psi))\,
\frac{ds}{(1-s)s^{\sigma}}\,d\psi
}
 &
  & \text{ if }u>1.
\end{align}
\end{proposition}

\begin{proof}
We first simplify the formula \eqref{Fade away} in a manner that is standard for Calder\'on--Zygmund singular integral
operators, see
\cite[Section II.3.4]{S70}.
That is, we use (i) pointwise bound, (ii) gradient bound, and (iii) cancellation property of the kernel \eqref{eq: Opel Kadett} to obtain:\\
\vskip 10pt
\ \hskip -25pt
$
{\displaystyle
(\cR_{-k}\Phi_\varepsilon)(u)
=
\left\{
\begin{array}{c}
\ \\
\ \\
\ \\
\ \\
\ \\
\ \\
\ \\
\ \\
\ \\
\end{array}
\right.
}
$
\vskip - 131pt
\begin{align}
\label{Crasborn}
&\ \hskip 50pt
\renewcommand{\arraystretch}{2.6}
\begin{tabular}{@{}l@{}}
${\displaystyle \int_{\{0<|w|<\delta_u\}}
\frac{(w/|w|)^{k}}{|w|}\cdot\frac{|u-w|^{-2\sigma}-|u|^{-2\sigma}}{|w|}
\,dA(w)}$ \\
${\displaystyle \hskip 20pt+
\int_{\{|w|>\delta_u\}\cap\{|w-u|<1\}}\frac{(w/|w|)^{k}}{|w|^2}|u-w|^{-2\sigma}\,dA(w)}$
\end{tabular}
\renewcommand{\arraystretch}{1}
\
&  & \text{ if }0<u<1,
\\
& & & \nonumber \\
\label{Odresitev}
&\ \hskip 50pt {\displaystyle \int_{\{|w-u|<1\}}\frac{(w/|w|)^{k}}{|w|^2}|u-w|^{-2\sigma}\,dA(w)}
&   & \text{ if }u>1.
\end{align}

In \eqref{Crasborn} one can take $\delta_u$ to be a half of the distance from $u$ to the set of discontinuities of $\Phi_\varepsilon$, which is the union of the unit circle and the origin. That is, $\delta_u=(1/2)\min\{u,|u-1|\}$. Yet $\delta_u$ will soon disappear anyway, as we develop more explicit ways to express \eqref{Crasborn} and \eqref{Odresitev}. Let us specify this next.

\smallskip
{\bf First we assume that \framebox{$0<u<1$}.}
For any $\varphi\in[0,2\pi)$, the ray $\mn{re^{i\varphi}}{r>0}$ hits the circle $\{w\in\mathbb{C}; |w-u|=1\}$, centered at $u$ and of radius $1$, at $r=u\cos\varphi+\sqrt{1-(u\sin\varphi)^2}$. Therefore \eqref{Crasborn} gives,
after passing to polar coordinates $(r,\varphi)$ and setting $s=r/u$,
\begin{equation}
\label{Nick Anderson}
\aligned
(\cR_{-k}\Phi_\varepsilon)(u)
&=
u^{-2\sigma}
\int_0^{2\pi}
e^{ik\varphi}
\int_0^{\delta_u/u}
\frac{|1-se^{i\varphi}|^{-2\sigma}-1}{s}
\,ds\,d\varphi\\
&\hskip 20pt+
u^{-2\sigma}
\int_0^{2\pi}
e^{ik\varphi}\int_{\delta_u/u}^{\cos\varphi+\sqrt{1/u^2-\sin^2\varphi}}
\frac{|1-se^{i\varphi}|^{-2\sigma}}{s}\,ds\,d\varphi.
\endaligned
\end{equation}
Because
$
|1-se^{i\varphi}|^{2}=1-2s\cos\varphi+s^2
$ only depends on $\cos\varphi$,
the inner integrands are $2\pi$-periodic and even with respect to $\varphi$, which means that instead of $\int_0^{2\pi}e^{i k \varphi}\hdots$ we may write $2\int_0^{\pi}\cos(k \varphi)\hdots$ in both summands from \eqref{Nick Anderson} above.
Note the oscillatory nature of these integrals, due to the presence of $\cos k\varphi$.

Using the change of variable $\psi=\pi-\varphi$ in $\int_{\pi/2}^\pi$ gives, for $k$ that is {\it odd} and an integrable $Z$,
\begin{equation}
\label{3AM}
\int_0^\pi Z(\cos\varphi)\cos k\varphi\, \,d\varphi
= \int_0^{\pi/2} [Z(\cos\varphi)-Z(-\cos\varphi)]\cos k\varphi\, \,d\varphi.
\end{equation}

We apply the formula \eqref{3AM} to both integrals from \eqref{Nick Anderson}, since this will give us
positivity of the first integrand there, along side with
cancellation that will eliminate $\delta_u/u$ as a limit point of some of the integrals, as shown next:
\begin{equation*}
\label{Mahalia - In the Upper Room}
\aligned
\frac{(\cR_{-k}\Phi_\varepsilon)(u)}{2u^{-2\sigma}}
& =
\int_0^{\pi/2}
\cos k\varphi
\int_0^{\delta_u/u}
\frac{
|1-se^{i\varphi}|^{-2\sigma}
-
|1+se^{i\varphi}|^{-2\sigma}}{s}
\,ds
\,d\varphi\\
&\hskip 15pt
+
\int_0^{\pi/2}
\cos k\varphi
\bigg(
\int_{\delta_u/u}^{\sqrt{1/u^2-\sin^2\varphi}+\cos\varphi}
\frac{|1-se^{i\varphi}|^{-2\sigma}}{s}\,ds
\\
&\hskip 79.5pt
-\int_{\delta_u/u}^{\sqrt{1/u^2-\sin^2\varphi}-\cos\varphi}
\frac{|1+se^{i\varphi}|^{-2\sigma}}{s}\,ds
\bigg)\,d\varphi.
\endaligned
\end{equation*}
Together with
\eqref{Scherbakov Symphony no.5}
this proves \eqref{Kalinnikov Symphony no.1}.

\smallskip
{\bf Now assume that} \framebox{$u>1$}.
Recall the formula \eqref{Odresitev}.
If $u>1$ then the ray $\mn{re^{i\varphi}}{r>0}$ hits the circle $\{w\in\mathbb{C}; |w-u|=1\}$, centered at $u$ and of radius $1$, only for $|\varphi|\leq \arcsin 1/u$, and that happens at $r=u\cos\varphi\pm\sqrt{1-(u\sin\varphi)^2}$. Therefore, by adapting the steps from the calculation in the previous case ($0<u<1$), we continue as
$$
\aligned
(\cR_{-k}\Phi_\varepsilon)(u)
&=\int_{-\arcsin 1/u}^{\arcsin 1/u}
e^{ik\varphi}\int_{u\cos\varphi-\sqrt{1-u^2\sin^2\varphi}}^{u\cos\varphi+\sqrt{1-u^2\sin^2\varphi}}
\frac{dr}{r|u-re^{i\varphi}|^{2\sigma}}\,d\varphi\\
&
=
2u^{-2\sigma}
\int_{0}^{\arcsin 1/u}
\cos k\varphi
\int_{-\sqrt{1/u^2-\sin^2\varphi}}^{\sqrt{1/u^2-\sin^2\varphi}}
\frac{dt}{(t+\cos\varphi)(t^2+\sin^2\varphi)^{\sigma}}\,d\varphi
\\
&=
4u^{-2\sigma}
\int_{0}^{1/u}
\int_{0}^{\arcsin\sqrt{1/u^2-t^2}}
\frac{\cos\varphi\cos k\varphi
}{(1-(t^2+\sin^2\varphi))(t^2+\sin^2\varphi)^{\sigma}}\,d\varphi\,dt\\
&=
4u^{-2\sigma}
\int_{0}^{1/u}
\int_{0}^{\sqrt{1/u^2-t^2}}
\frac{\cos(k\arcsin x)
}{(1-(t^2+x^2))(t^2+x^2)^{\sigma}}\,dx\,dt\\
&=
4u^{-2\sigma}
\iint_{B_{\mathbb{R}^2}^{++}(0,1/u)}
\frac{\cos(k\arcsin x)
}{(1-(x^2+y^2))(x^2+y^2)^{\sigma}}\,dx\,dy\\
&=
2u^{-2\sigma}
\int_0^{\pi/2}
\int_0^{1/u^2}
\cos(k\arcsin(\sqrt{s}\cos\psi))\,
\frac{ds}{(1-s)s^{\sigma}}\,d\psi.
\endaligned
$$
Here $B_{\mathbb{R}^2}^{++}(0,1/u)$ is the ball in $\mathbb{R}^2$, centered at the origin and of radius $1/u$, intersected with the first quadrant in $\mathbb{R}^2$.
\end{proof}

\subsection{Summary so far}
By means of
Proposition~\ref{Markovina Velikic}
we reformulate our main goal, most recently expressed as \eqref{Kronos}.
We prove the following two results which
through
\eqref{Utushka lugovaya}, \eqref{Kalinnikov Symphony no.1} and \eqref{Braun HC5030}
together
combine into \eqref{Kronos}:
\begin{theorem}
\label{metatarsus dex}
Denote $\sigma=1/p-\varepsilon$ as in \eqref{19vodki}.
For odd $k\in\mathbb{N}$, $k\geq3$, and $p\geq3$,
$$
\lim
_{\varepsilon\searrow0}
\left(
\int_1^\infty |G_{k,\sigma}(u)|^p\cdot \varepsilon u^{-1+2\varepsilon p}\,du
\right)^{1/p}
=0.
$$
\end{theorem}

\begin{theorem}
\label{gastrocnemius}
Denote $\sigma=1/p-\varepsilon$ as in \eqref{19vodki}.
For odd $k\in\mathbb{N}$, $k\geq3$, and $p\geq3$,
$$
\limsup_{\varepsilon\searrow0}
\left(
\int_0^1 |G_{k,\sigma}(u)|^p\cdot \varepsilon u^{-1+2\varepsilon p}\,du
\right)^{1/p}
\leqsim\frac{1}{k^{2(1-1/p)}(p-1)}.
$$
\end{theorem}
\noindent
This is now our objective.
We will prove Theorems \ref{metatarsus dex}, \ref{gastrocnemius} in Sections  \ref{ke se prestoram} and \ref{Zemi ogin}, respectively.

\subsection{Integral asymptotics}
We will need the following two auxiliary results.

\begin{lemma}
\label{Barney Miller - Aces of Fashion}
Uniformly in $\beta,\eta\in(0,1)$ we have
\begin{equation}
\label{Bach Musette BWV Anh 126}
\int_0^{\eta}\frac{dv}{(1-v)v^\beta}
\,\sim\,
\frac{\eta^{1-\beta}}{1-\beta}+\log\frac1{1-\eta}.
\end{equation}
\end{lemma}

\begin{remark}
Note that this can be interpreted as a sort of a ``log-product rule for integrals''.
That is, \eqref{Bach Musette BWV Anh 126} can be rewritten as
$
\int fg\,\sim\,\int f+\int g
$
for $f(v)=(1-v)^{-1}$ and $g(v)=v^{-\beta}$, and $\int\varphi=\int_0^\eta\varphi(v)\,dv$,
which reminds one of the well-known elementary rule $\log(fg)=\log f+\log g$.
\end{remark}

\begin{proof}
Consider separately the cases $0<\eta<1/2$ and $1/2<\eta<1$. In the latter case split further, namely divide the integration domain into $(0,1/4)\cup(1/4,\eta)$. We leave the details to the reader.
\end{proof}

\begin{lemma}
\label{sticky}
Fix $p>0$.
For $M> 1/p$ define
\begin{equation}
\label{In expecto Game 2 Denver:Lakers}
I_p(M):=\int_1^\infty\left(\log\frac{1}{1-x^{-M}}\right)^pdx.
\end{equation}
Then the integrals $I_p(M)$ converge
and
$I_p(M)\sim 1/M$ as $M\rightarrow\infty$.
\end{lemma}

\begin{proof}
Substituting with
$t=-\log(1-x^{-M})$
in \eqref{In expecto Game 2 Denver:Lakers} gives
$I_p(M)=J_p(1/M)/M$ for
$$
J_p(\varepsilon)
:=\int_0^\infty t^pe^{-t}(1-e^{-t})^{-1-\varepsilon}dt,
\hskip 30pt
0\leq\varepsilon<p.
$$
Observing that the function  $\varepsilon\mapsto (1-e^{-t})^{-1-\varepsilon}$ is increasing for any $t>0$ gives the estimate
$J_p(0)\leq J_p(\varepsilon)
\leq J_p(p/2)$ for $\varepsilon\in[0,p/2]$. The proof will be over once we show that the integrals $J_p(0),J_p(p/2)$ in fact converge (i.e., are both finite and nonzero), which is not difficult.
\end{proof}

\section{Proof of Theorem~\ref{metatarsus dex} (integration over $u>1$)}
\label{ke se prestoram}

Recall that $G_\sigma\big|_{(1,\infty)}$ is given by \eqref{Braun HC5030}. Take $u>1$. 
By first estimating crudely $|\cos|\leq1$ in \eqref{Braun HC5030}, and then applying Lemma~\ref{Barney Miller - Aces of Fashion}, we get
$$
|G_\sigma(u)|
\,\leqsim\,
\int_0^{1/u^{2}}
\frac{ds}{(1-s)s^{\sigma}}
\,\leqsim\,
\frac{u^{-2(1-\sigma)}}{1-\sigma}
+\log\frac1{1-1/u^{2}}\,.
$$
Therefore, since $-2(1-\sigma)p-1+2\varepsilon p=1-2p$, we may proceed as
\begin{equation*}
\int_{1}^\infty \big|G_\sigma(u)|^p\,\varepsilon u^{-1+2\varepsilon p}du
\,\leqsim_p\,
\varepsilon
\int_{1}^\infty u^{1-2p}\,du
+
\varepsilon\int_{1}^\infty
\left(
{\log}
\frac{1}{1-1/u^2}
\right)^p
 u^{-1+2\varepsilon p}du.
\end{equation*}
The middle expression clearly tends to zero as $\varepsilon\rightarrow0$.
So does the last one;
to see this, introduce the new variable $x=u^{2\varepsilon p}$ and apply Lemma~\ref{sticky}.
\qed

\section{Proof of Theorem~\ref{gastrocnemius} (integration over $0<u<1$)}
\label{Zemi ogin}

In accordance with \eqref{Kalinnikov Symphony no.1},
for
$u\in(0,1)$
and $j=0,1$ define $G_\sigma^j=G_{k,\sigma}^{j}$ by
\begin{equation}
\label{Kamchatka}
G_\sigma^j(u)
=\int_0^{\pi/2}
H_\sigma^j\left(\sqrt{1/u^2-1},\cos\varphi\right)
\cos k\varphi
\,d\varphi,
\end{equation}
where $H_\sigma^j=H_{k,\sigma}^j$ are given by \eqref{Scherbakov Symphony no.5}. Then
clearly
$G_\sigma\big|_{(0,1)}=G_\sigma^0+G_\sigma^1.$
\label{Shostakovich 1 Hahn}
Note that $k$ appears in \eqref{Kamchatka} only in the term $\cos k\varphi$.

\subsection{Estimates for $H_\sigma^1$ and $G_\sigma^1$}

\begin{proposition}
\label{Ella Fitzgerald - Caravan}
Assume that $0<\varepsilon<1/p\leq1/3$. For all $u\in(0,1)$ and $b\in[0,1]$ we have
\begin{equation*}
\label{In The Woods}
H_\sigma^1(\sqrt{1/u^2-1},b)\,\leqsim\,
\log\frac1{1-u}\,,
\end{equation*}
with the implied constant being absolute.
Here $\sigma=1/p-\varepsilon$, as in \eqref{19vodki}.
\end{proposition}

\begin{proof}
Write $a = \sqrt{1/u^2-1}$.
First observe that the integration domain of $H_\sigma^1(a,b)$, namely $[c-b,c+b]$, increases with $b$, and so does the integrand. Therefore it suffices to estimate
$$
H_\sigma^1(a,1)=\int_{1/u-1}^{1/u+1}\frac{ds}{s|1-s|^{2\sigma}} \,.
$$
We split cases. If
$0<u\leq 1/2$
then
$[1/u-1,1/u+1]\subset[1/(2u),2/u]$.
Consequently,
$$
H_\sigma^1(a,1)\sim u\int_{1/u-1}^{1/u+1}\frac{ds}{(s-1)^{2\sigma}}
=\frac{u^{2\sigma}}{1-2\sigma}\left(1-(1-2u)^{1-2\sigma}\right).
$$
For $x,\gamma\in(0,1)$ we have $1-x^\gamma\leq 1-x$, thus the above expression is controlled by $\leqsim u^{1+2\sigma}$.

If, on the other hand,
$1/2\leq u \leq 1$, then
$$
\aligned
H_\sigma^1(a,1)
& =\int_{1/u-1}^{1}\frac{ds}{s(1-s)^{2\sigma}}+\int_{1}^{1/u+1}\frac{ds}{s(s-1)^{2\sigma}}
\leq\int_{0}^{2-1/u}\frac{dv}{(1-v)v^{2\sigma}}+\int_{0}^{1/u}\frac{dw}{w^{2\sigma}}\,.
\endaligned
$$
Here we wrote $v=1-s$, $w=s-1$ and estimated $w+1\geq1$. Now the very last integral can be explicitly calculated, while for the penultimate one we use Lemma \ref{Barney Miller - Aces of Fashion} again.

At the end we observe that on $(0,1)$ we have $u^{1+2\sigma}\leq u\leq \log(1/(1-u))$.
\end{proof}

\begin{corollary}
\label{G1,2}
We have
$$
\lim_{\varepsilon\searrow0}
\int_0^1|G_{\sigma}^1(u)|^p\cdot \varepsilon u^{-1+2\varepsilon p}\,du
=0.
$$
\end{corollary}

\begin{proof}
From \eqref{Kamchatka} and Proposition~\ref{Ella Fitzgerald - Caravan} we immediately get
$$
\int_0^1|G_{\sigma}^1(u)|^p\cdot\varepsilon u^{-1+2\varepsilon p}du
\leqsim
\int_0^1\left(\log\frac{1}{1-u}\right)^p\cdot\varepsilon u^{-1+2\varepsilon p}du.
$$

Introducing new variables, first $v=u
^{2\varepsilon p}$ and then $x=1/v$, turns the last integral into
$$
\frac{1}{2p}\int_1^\infty\left(\log\frac{1}{1-x^{-1/(2\varepsilon p)}}\right)^p\,\frac{dx}{x^2}
\leq
\frac{1}{2p}\int_1^\infty\left(\log\frac{1}{1-x^{-1/(2\varepsilon p)}}\right)^p dx.
$$
Now Lemma~\ref{sticky} shows that this tends to zero as $\varepsilon\searrow0$.
\end{proof}

\subsection{Estimates for $H_\sigma^0$ and $G_\sigma^0$}
Recall the notation $\sigma=1/p-\varepsilon$, cf. \eqref{19vodki}. By combining \eqref{Scherbakov Symphony no.5} and
\eqref{Kamchatka} we get, for $u\in(0,1)$,
\begin{equation}
\label{Sarmati}
G_\sigma^0(u)
=\int_0^{\pi/2}
\int_0^{\sqrt{1/u^2-\sin^2\varphi}-\cos\varphi}
\frac{(1-2s\cos\varphi+s^2)^{-\sigma
}-(1+2s\cos\varphi+s^2)^{-\sigma
}}{s}
\,ds
\cos k\varphi
\,d\varphi.
\end{equation}

Let $\mu_\varepsilon=\mu_{\varepsilon,p}$ be the measure on $(0,1)$ defined by
\begin{equation}
\label{mue}
d\mu_\varepsilon(u):=\varepsilon u^{-1+2\varepsilon p}du.
\end{equation}
From $G_\sigma=G_\sigma^0+G_\sigma^1$, cf. page \pageref{Shostakovich 1 Hahn}, and
Corollary~\ref{G1,2}
we see that
in order to finish the proof of Theorem~\ref{gastrocnemius}, we have to prove that
\begin{equation}
\label{I want to be ninja}
\limsup_{\varepsilon\searrow0}\nor{G_{k,\sigma}^0}_{L^p(\mu_\varepsilon)}\leqsim\frac{1/p}{k^{2(1-1/p)}}
\end{equation}
uniformly in $p\geq3$ and odd $k\in\mathbb{N}\backslash\{1,2\}$, cf. \eqref{kp3}.

\subsection{Oscillation}
So far we have not yet exploited the oscillating
nature of the integrand in \eqref{Sarmati} coming from $\cos k\varphi$. In order to do it now, we
split the outer integration domain in \eqref{Sarmati}, that is, $[0,\pi/2)$, into $k$ disjoint intervals $I_j$ of size $\pi/(2k)$ each.
Thus we decompose
$[0,\pi/2)=I_1\sqcup I_2\sqcup\hdots\sqcup I_k$
with
$I_j=\left[(j-1)\pi/(2k),j\pi/(2k)\right)$.
Recall that $k\in\mathbb{N}$, $k\geq3$, was {\it odd} and define
\begin{equation*}
l=l(k):=\frac{k-1}{2}.
\end{equation*}
We expect
that when integrating over
$J_m:=I_{2m-1}\sqcup I_{2m}$, an {\it almost cancellation} should occur, since
the sign of $\cos k\varphi$
will be positive on one of the intervals $I_{2m-1},I_{2m}$ and negative on the other. So we want to group the intervals as follows:
\begin{equation*}
\underbrace{I_1\sqcup I_2}_{J_1}\sqcup\underbrace{I_3\sqcup I_4}_{J_2}\sqcup\hdots\sqcup\underbrace{I_{k-2}\sqcup I_{k-1}}_{J_{l(k)}}\sqcup I_{k}.
\end{equation*}
More precisely,
our plan regarding the integration with respect to $d\varphi$ in \eqref{Sarmati} reads
\begin{equation}
\label{Parsifal Ouverture}
G_\sigma^0(u)=
\int_0^{\pi/2}
R(\varphi)\cos k\varphi\,d\varphi
=\int_{J_1}+\sum_{m=2}^{l(k)}\int_{J_m}+\int_{I_k}.
\end{equation}
(If $k=3$ and thus $l(k)=1$, the middle sum disappears.)
The optimal bounds are eventually attained for $\varphi\approx0$, that is, on $J_{1,2}$; see Section \ref{Cessna}.

Let us start with the calculations.
Introducing the new variable
$(2m-1)\pi/k-\varphi$ in the integral over $I_{2m}$ in the left-hand side below gives
$$
\aligned
\Big(\int_{I_{2m-1}}+\int_{I_{2m}}\Big)R(\varphi)\cos k\varphi\,d\varphi
&= \int_{I_{2m-1}}\Big[R(\varphi)-R\Big((2m-1)\frac{\pi}{k}-\varphi\Big)\Big]\cos k\varphi\,d\varphi.
\endaligned
$$
The observation $[0,\pi/2)=kI_{2m-1}-(m-1)\pi$ gives rise to the new variable $\vartheta=k\varphi-(m-1)\pi$. Finally introduce $\psi=\pi/2-\vartheta$, which gives
\begin{align*}
& \int_{J_m}R(\varphi)\cos k\varphi\,d\varphi \\
&= \frac{(-1)^{m-1}}{k}\int_{0}^{\pi/2}
\Big[R\Big(\frac{(2m-1)\pi/2-\psi}{k}\Big)-R\Big(\frac{(2m-1)\pi/2+\psi}{k}\Big)\Big]\sin \psi\,d\psi.
\end{align*}

Similarly, for the integral over $I_k$ we have
$$
\int_{I_k}R(\varphi)\cos k\varphi\,d\varphi
=\frac{(-1)^{l(k)}}{k}\int_0^{\pi/2}
                       R\Big(\frac{\pi}{2}-\frac{\psi}{k}\Big)
\sin \psi\,d\psi
\,.
$$

Putting everything together and applying Lagrange's mean value theorem, for some
\begin{equation}
\label{EU}
\xi=\xi_m=\xi(\psi,m,k)\in\Big(\frac{(2m-1)\pi/2-\psi}{k},\frac{(2m-1)\pi/2+\psi}{k}\Big)
\end{equation}
the formula \eqref{Parsifal Ouverture} now looks like
\begin{equation}
\label{Gre zares}
G_\sigma^0(u) =    \frac{2}{k^2}\sum_{m=1}^{l(k)}
(-1)^{m}      \int_{0}^{\pi/2}
       R'(\xi_m) \psi \sin\psi \,d\psi +\frac{(-1)^{l(k)}}{k}\int_0^{\pi/2}
R\Big(\frac{\pi}{2}-\frac{\psi}{k}\Big)\sin \psi\,d\psi.
\end{equation}

As hinted at above, we plan to use \eqref{Gre zares} for
$R(\varphi) =H_\sigma^0\big(\sqrt{1/u^2-1},\cos\varphi\big)$,
where $H_\sigma^0$ was defined in \eqref{Scherbakov Symphony no.5}; recall that $\sigma=1/p-\varepsilon\in(0,1/3]$. Therefore we have the formula
\begin{equation}
\label{Kaede}
R'(\varphi)
=T(\cos\varphi)\sin\varphi,
\end{equation}
where
\begin{equation}
\label{pehtran}
T(b)=T(a,b)
: = -\frac{\pd}{\pd b}H_\sigma^0(a,b) = U (a,b)-V(a,b)
\end{equation}
for
\begin{eqnarray}
\label{T1}
&&   {\displaystyle
U (a,b)
=\frac{ [a^2+1-4b(\sqrt{a^2+b^2}-b)]^{-\sigma}-(a^2+1)^{-\sigma} }{\sqrt{a^2+b^2}}},\\
\label{T2}
&&  {\displaystyle
V(a,b)
= 2\sigma\int_0^{\sqrt{a^2+b^2}-b}
\left[
(s^2-2bs+1)^{-\sigma-1}+(s^2+2bs+1)^{-\sigma-1}
\right]
ds.}
\end{eqnarray}

We merge \eqref{Gre zares}, \eqref{Kaede} and \eqref{pehtran} into the following finding.
Recall that $\xi_m$ is constrained by the requirement \eqref{EU}.
Moreover, our choice of $R$ also depends on $u$ and $\sigma=1/p-\varepsilon$, therefore we have have to interpret \eqref{EU} as $\xi_m=\xi_m(\psi,k,u,\sigma)$.

\begin{proposition}
\label{Jordan K'nev}
Let $U,V$ be given as in \eqref{T1}, \eqref{T2}, respectively, and $\xi_m$ as in \eqref{EU}.
Then $G_\sigma^0=\cU-\cV+\cW$, where for $u\in(0,1)$ we have
\begin{align}
\label{U}
\cU(u)=\cU_\sigma(u)&=
\frac{2}{k^2}\sum_{m=1}^{l(k)}
(-1)^{m}      \int_{0}^{\pi/2}
      U\left(\sqrt{1/u^2-1},\cos\xi_m\right)\sin\xi_m\cdot \psi\sin\psi \,d\psi, \\
\label{V}
\cV(u)=\cV_\sigma(u)&=
\frac{2}{k^2}\sum_{m=1}^{l(k)}
(-1)^{m}      \int_{0}^{\pi/2}
      V\left(\sqrt{1/u^2-1},\cos\xi_m\right)\sin\xi_m\cdot \psi\sin\psi \,d\psi, \\
\cW(u) =\cW_\sigma(u)&=
              \frac{(-1)^{l(k)}}{k}\int_0^{\pi/2}
H_\sigma^0\left(\sqrt{1/u^2-1},\sin(\psi/k)\right)\cdot\sin\psi\,d\psi.\nonumber
\end{align}
\end{proposition}

\begin{proposition}
\label{Padna}
Assuming the notation of Proposition~\ref{Jordan K'nev}, we have
\begin{align}
\label{Uest}
\lim_{\e\searrow0}\nor{\cU}_{L^p(\mu_\e)} & =0,\\
\label{Vest}
\limsup_{\e\searrow0}\nor{\cV}_{L^p(\mu_\e)} & \leqsim\frac{1/p}{k^{2(1-1/p)}},\\
\label{West}
\limsup_{\e\searrow0}\nor{\cW}_{L^p(\mu_\e)} & \leqsim\frac{1/p}{k^{2}}.
\end{align}
\end{proposition}
Together, Propositions~\ref{Padna},~\ref{Jordan K'nev} imply \eqref{I want to be ninja}, which thus completes the proof of Theorem~\ref{gastrocnemius} (and hence Theorems~\ref{t: Lowest} and \ref{t: main}).

In the consecutive sections we prove statements \eqref{Uest}, \eqref{Vest} and \eqref{West} one by one.
As indicated in \eqref{Parsifal Ouverture}, it will be useful to single out the term in \eqref{Gre zares} corresponding to $m=1$.

\subsection{Proof of (\ref{Uest})}

Notice that, for $c=\sqrt{a^2+b^2}$ and $\Upsilon(a,b)=(c-2b)^2+(1-b^2)$,
\begin{equation}
\label{Hamilkar}
U(a,b)=\frac{\Upsilon(a,b)^{-\sigma}-\Upsilon(a,0)^{-\sigma}}{c}\,.
\end{equation}

\begin{lemma}
\label{cvicek Martincic}
Let $u\in(0,1)$, $a=\sqrt{1/u^2-1}$, $b\in(0,1)$, $\sigma\in(0,1/2)$ and
$\mu=1-|1-2u|^{2\sigma}$.
Then
$$
\Upsilon(a,b)^{-\sigma}-\Upsilon(a,0)^{-\sigma}\leq\mu\Upsilon(a,b)^{-\sigma}\,.
$$
\end{lemma}

\begin{proof}
The targeted inequality is equivalent to
\begin{equation}
\label{Majke - Hajde plesi sa mnom}
(1-\mu)^{1/\sigma}
\leq\frac{\Upsilon(a,b)}{\Upsilon(a,0)}
=1-\frac{4(1-u^2)}{\sqrt{(a/b)^2+1}+1}
\hskip 30pt
\forall b\in(0,1).
\end{equation}
Clearly, the right-hand side of \eqref{Majke - Hajde plesi sa mnom}
is decreasing as a function of $b$. Therefore its infimum on $(0,1]$ is attained at $b=1$,
where it takes the value $(1-2u)^2$.
\end{proof}

\begin{proposition}
\label{Confutatis}
Let $U$ be as in \eqref{T1} and
$u,a,b,\sigma,\mu$ as in Lemma~\ref{cvicek Martincic}.
Then we have
${\displaystyle
U (a,b)\leqsim u^2(1-b^2)^{-\sigma}.
}$
The implied constant is absolute.
\end{proposition}

\begin{proof}
Write $c=\sqrt{a^2+b^2}$.
From \eqref{Hamilkar} and Lemma~\ref{cvicek Martincic} we obtain
\begin{equation}
\label{eq: Naqsh-e Jahan}
U(a,b)\leqslant\frac{\mu}{c}\,\Upsilon(a,b)^{-\sigma}\leq\frac{\mu}{c}\,\big(1-b^2\big)^{-\sigma}\,.
\end{equation}
Now we separately estimate $\mu$ and $c$.

For $x,\gamma\in(0,1)$ we have $1-x^\gamma\leq 1-x$, therefore
\begin{equation*}
\mu=1-|1-2u|^{2\sigma}\leq1-|1-2u|=2
	\left\{
		\begin{array}{ccl}
			u & ; & u\in(0,1/2]\\
			1-u & ; & u\in[1/2,1).
		\end{array}
	\right.
\end{equation*}
As to $c$, we use that
$c\geq a=\sqrt{1-u^2}/u\geq\sqrt{1-u}/u$.

Putting these facts together we obtain
$$
\frac{\mu}{c}\leqsim
	\left\{
		\begin{array}{ccl}
			u^2/\sqrt{1-u} & ; & u\in(0,1/2]\\
			u\sqrt{1-u} & ; & u\in[1/2,1)
		\end{array}
	\right\} \leqsim u^2.
$$
Combined with \eqref{eq: Naqsh-e Jahan} this settles the proof.
\end{proof}

\begin{proof}[Proof of (\ref{Uest})]
We will consider each summand in \eqref{U} separately, that is, we decompose $\cU=\cU_1+\hdots+\cU_{l(k)}$, where
for a fixed $m\in\{1,\hdots,l(k)\}$ we denote
$$
{\cU}_m(u)=
\frac{2}{k^2}(-1)^{m}
      \int_{0}^{\pi/2}
       U\left(\sqrt{1/u^2-1},\cos\xi_m\right)\sin\xi_m\cdot\psi\sin\psi \,d\psi.
$$
Proposition~\ref{Confutatis} enables us to estimate
$$
|{\cU}_m(u)|\leqsim
\left(\frac{u}{k}\right)^2
\int_{0}^{\pi/2}
      (\sin\xi_m)^{1-2\sigma} \,d\psi
      \leqsim u^2.
$$
Consequently, recalling \eqref{mue} and
$\sigma=1/p-\varepsilon$ once again,
we obtain
$$
\nor{\cU_m}_{L^p(\mu_\varepsilon)}^p
=\int_0^{1}|\cU_m (u)|^p\cdot\varepsilon u^{-1+2\varepsilon p}\,du
\leqsim
\varepsilon\int_0^{1}u^{2p-1+2p\e}\,du
\leq\varepsilon\int_0^{1}u^{2p-1}\,du.
$$
Obviously, the expression on the right-hand side converges to zero as $\varepsilon\rightarrow0$.
\end{proof}

\subsection{Proof of (\ref{Vest})}
\label{Cessna}
As in the case of $\cU$, we start with a pointwise estimate of the integrand of $\cV$ (that is, $V$).

\begin{proposition}
\label{Prokofiev - Sonata no. 6}
Let $V$ be as in \eqref{T2}.
For
$a>0$ and
$\xi\in(0,\pi/2)$
we have
\begin{equation*}
\label{Offertorium}
V(a,\cos\xi)
\,\leqsim\,
\frac\sigma{(\sin\xi)^{1+2\sigma}}
\,.
\end{equation*}
\end{proposition}

\begin{proof}
Denote $b=\cos\xi$.
As before, set $c=\sqrt{a^2+b^2}$.
Since the second summand in \eqref{T2} is positive and, at the same time, dominated by the first summand, we see that
$$
V(a,b)
\sim\sigma
\int_0^{c-b}\left[(s-b)^2+(1-b^2)\right]^{-1-\sigma}ds.
$$
From $(x+y)^2\sim x^2+y^2$ for $x,y>0$, we now get
$$
V(a,b)
\sim\sigma
\int_0^{c-b}\left(|s-b|+\sqrt{1-b^2}\right)^{-2\sigma-2}ds.
$$
At this point consider separately the cases $c-b\leq b$ and $c-b\geq b$. In both of these cases, the above integral can be explicitly calculated. This quickly leads to the desired conclusion.
\end{proof}

\begin{proof}[Proof of (\ref{Vest})]
Once again write $a=\sqrt{1/u^2-1}$.
For any $\psi\in(0,\pi/2)$ and $m\in\{1,\hdots,l(k)\}$ let $\xi_m=\xi_m(\psi,k,u,\sigma
)$ be defined in accordance with \eqref{EU}.
Here we decompose \eqref{V} as $\cV=\cV_1+\cV_{\text{sum}}$, where
\begin{align}
\label{Blue Morgan}
{\cV}_1(u)&=
-\frac{2}{k^2}\int_{0}^{\pi/2}
      V\left(\sqrt{1/u^2-1},\cos\xi_1\right)\sin\xi_1\cdot \psi\sin\psi \,d\psi\\
\label{Vsum}
{\cV}_{\text{sum}}(u)&=
\frac{2}{k^2}
      \int_{0}^{\pi/2}
\sum_{m=2}^{l(k)}
   (-1)^{m}
       V\left(\sqrt{1/u^2-1},\cos\xi_m\right)\sin\xi_m
       \cdot
       \psi\sin\psi \,d\psi.
\end{align}
(If $k=3$, we just leave $\cV=\cV_1$.)

\medskip
\noindent
\framebox{The ${\cV_1}$ part}.
\label{Omodaka Hikokuro}
From \eqref{EU} we get $\sin\xi_1\sim\xi_1\geq(\pi/2-\psi)/k$, which in combination with \eqref{Blue Morgan} and
Proposition \ref{Prokofiev - Sonata no. 6}
yields
$$
|\cV_{1}(u)|
\leqsim\frac{\sigma }{k^2}\int_0^{\pi/2}
\left(\frac{k}{\pi/2-\psi}\right)^{2\sigma}d\psi
\leqsim
\frac{\sigma }{(1-2\sigma)k^{2(1-\sigma)}}
\leqsim
\frac{\sigma }{k^{2(1-\sigma)}}.
$$
In the last inequality we used that $p\geq3$, see \eqref{kp3}, which translates into $1-2\sigma\sim1$.
Hence
$$
\limsup_{\varepsilon\searrow0}
\nor{\cV_{1}}_{L^p(\mu_\varepsilon)}
\leqsim\frac{1/p }{k^{2(1-1/p)}}\,.
$$
This is exactly the estimate we were trying to get, cf. \eqref{I want to be ninja} and \eqref{Vest}.

\smallskip
\noindent
\framebox{The ${\cV_{\text{sum}}}$ part}.
We want to estimate $\nor{\cV_{\text{sum}}}_{L^p(\mu_\varepsilon)}$.
Unlike before, we do not want to estimate each term separately. Instead, we will once again use alternation. The idea is simply that if $m\mapsto V(a,\cos\xi_m)\sin\xi_m$
is a {\it decreasing} function, then the modulus of the {\it alternating} sum from \eqref{Vsum} is majorized by the modulus of its first term
(i.e., the one corresponding to $m=2$) only.
So we would only have to estimate {\it one} term, instead of an increasingly large (in terms of the number of summands) sum.
In other words, we would immediately get
\begin{equation}
\label{Damietta}
|{\cV}_{\text{sum}}(u)|\leqsim
\frac{1}{k^2}
      \int_{0}^{\pi/2}
       V(a,\cos\xi_2)\sin\xi_2\,d\psi
       =:\widetilde{\cV}_{2}(u).
\end{equation}

Since $\xi_m$ increases with $m$, while the cosine decreases on $(0,\pi/2)$, it is enough
to show that the function $\Lambda: b\mapsto V(a,b)\sqrt{1-b^2}$ is increasing.
In fact it is; the next result makes that statement precise.

\begin{proposition}
\label{Sekigahara}
Fix $\sigma_0\in(0,1/2)$ and for a chosen $\sigma\in(\sigma_0/2,\sigma_0)$ let $V$ be as in \eqref{T2}.
For every $\omega\in(0,1)$ there exists $M=M(\sigma_0,\omega)>2$ such that for all $a>M$ and all $\sigma$ as above,
the function $X=X_{a}$, defined as
$$
X_a(b)=V(a,b)\sqrt{1-b^2},
$$
increases on
$(\omega,1)$.
\end{proposition}

\begin{proof}
It suffices
to show that
\begin{equation}
\label{sequentia}
Y_a(b)
:=\frac{X_a'(b)\sqrt{1-b^2}}{2\sigma}
=\frac{1}{2\sigma}
\left(
(1-b^2)\frac{\pd V}{\pd b}(a,b)-bV(a,b)
\right)
\end{equation}
is a {\it positive} function, under the conditions specified in the formulation of the Proposition.

Introduce the {\it ad hoc} notation
$$
\alpha_\pm(s)=s^2\pm 2bs+1.
$$
We will use that
\begin{equation}
\label{eq: travel dunk}
\frac{\pd}{\pd s}\alpha_\pm(s)^{-\sigma-1} =-2(1+\sigma)(s\pm b)\alpha_\pm(s)^{-\sigma-2}.
\end{equation}

Write $c=\sqrt{a^2+b^2}$.
From \eqref{sequentia}, \eqref{eq: travel dunk} and \eqref{T2} we
obtain, after some calculation,
$$
\aligned
Y_a(b)
&=
(1-b^2)\Big(\frac{b}{c}-1\Big)
\left[
\alpha_-(c-b)^{-\sigma-1}+\alpha_+(c-b)^{-\sigma-1}
\right]
\\
&\hskip 10pt
-\int_0^{c-b}
\frac{\pd}{\pd s}
\left[
(bs+1)\alpha_-(s)^{-\sigma-1}
+
(bs-1)\alpha_+(s)^{-\sigma-1}
\right]ds
\\
&\hskip 20pt
+2(1+\sigma)b
\int_0^{c-b}
\left[
\alpha_-(s)^{-\sigma-2}+\alpha_+(s)^{-\sigma-2}
\right](1-s^2)\,
ds.
\endaligned
$$
Introduce another piece of notation, namely,
$$
Z_a(b):=(c-b)(c-b+b^{-1})c^{-1}=c+1/b-\Big(2b+\frac{1-b^2}c\Big)\,,
$$
so as to proceed as
$$
\aligned
Y_a(b)
& =
2(1+\sigma)b
\int_0^{c-b}
\left[\alpha_+(s)^{-\sigma-2}+\alpha_-(s)^{-\sigma-2}\right](1-s^2)\,ds\\
&\hskip 25pt
-
\left[
bZ_a(b)+1
\right]
\alpha_-(c-b)^{-\sigma-1}
-
\left[
bZ_a(b)-1
\right]
\alpha_+(c-b)^{-\sigma-1}.
\endaligned
$$
If $a$ is large, then $bZ_a(b)-1=[(c-b)^2-1]b/c>0$.
Since also
$\alpha_-(c-b)^{-\sigma-1}\geq\alpha_+(c-b)^{-\sigma-1}$
and $1+\sigma\geq1$,
in order to prove $Y_a(b)\geq0$ it suffices to show that, for large $a$ and $b\in(\omega,1)$,
\begin{equation}
\label{Shadi Khries}
\int_0^{c-b}
\left(\alpha_+(s)^{-\sigma-2}+\alpha_-(s)^{-\sigma-2}\right)(1-s^2)\,ds
\geq
Z_a(b)
\alpha_-(c-b)^{-\sigma-1}.
\end{equation}
Because large $a$ implies $c-b>1$,
we may write the above integral as
$
\int_0^{1/(c-b)}+\int_{1/(c-b)}^1+\int_1^{c-b}.
$
By choosing $w=1/s$
in the $\int_1^{c-b}$ part
and observing that $\alpha_\pm(1/w)=\alpha_\pm(w)/w^2$,
we obtain
\begin{equation}
\label{positive integrand}
\aligned
\int_0^{c-b}
\frac{1-s^2}{\alpha_\pm(s)^{\sigma+2}}\,ds
& =
\int_0^{1/(c-b)}
\frac{1-s^2}{\alpha_\pm(s)^{\sigma+2}}\,ds
+\int_{1/(c-b)}^{1}
\frac{(1-s^2)(1-s^{2\sigma})}{\alpha_\pm(s)^{\sigma+2}}\,ds\\
&\geq
\int_{0}^{1}
\frac{(1-s^2)(1-s^{2\sigma})}{\alpha_\pm(s)^{\sigma+2}}\,ds
\geq
\int_{0}^{1}
\frac{(1-s^2)(1-s^{\sigma_0})}{\alpha_\pm(s)^{\sigma+2}}\,ds.
\endaligned
\end{equation}
For the last inequality we needed the assumption $0<\sigma_0/2<\sigma$.

Now we are ready for a simplification of \eqref{Shadi Khries}. On the left-hand side of \eqref{Shadi Khries}:
\begin{itemize}
\item
first apply \eqref{positive integrand} to the $\alpha_-,\alpha_+$ parts alike;
\item
then apply the convexity of $x\mapsto x^{-\sigma-2}$ which gives, for $s\in(0,1)$,
$$
\alpha_+(s)^{-\sigma-2}+\alpha_-(s)^{-\sigma-2}\geq2(1+s^2)^{-\sigma-2}\geq2^{-\sigma-1}\geq2^{-\sigma_0-1}.
$$
\end{itemize}

On the right-hand side of \eqref{Shadi Khries} use that:
\begin{itemize}
\item
$Z_a(b)\leq c+1/b\leq a\sqrt{2}+1/\omega$,
where we recalled that $c^2=a^2+b^2\leq a^2+1\leq 2a^2$;
\item
$\alpha_-(c-b)\geq(a-2)^2$ for large $a$, therefore $\alpha_-(c-b)^{-\sigma-1}\leq(a-2)^{-\sigma_0-2}$.
\end{itemize}
Putting all this together,
we see that
the inequality \eqref{Shadi Khries} would follow if we established
\begin{equation*}
2^{-\sigma_0-1}
\int_{0}^{1}
(1-s^2)(1-s^{\sigma_0})\,ds
\geq
\frac{a\sqrt{2}+1/\omega}{(a-2)^{2+\sigma_0}}.
\end{equation*}
Clearly, the left-hand side above is strictly positive and does {\it not} depend on $a$, while the right-hand side converges to zero as $a\rightarrow\infty$.
Of course, this means that the above inequality is fulfilled if $a>3$ is large enough.
\end{proof}

Let us now estimate $\nor{{\cV}_{\text{sum}}}_{L^p(\mu_\varepsilon)}$.
First note that, by \eqref{EU}, for $m\in\{2,\hdots,l(k)\}$ we have
\begin{equation}
\label{gaaaaaaa-gaaaaaaa}
\xi_m\in\left[\frac{\pi}{k},\Big(1-\frac1k\Big)\frac\pi2\right].
\end{equation}
By using Proposition \ref{Prokofiev - Sonata no. 6}
and \eqref{gaaaaaaa-gaaaaaaa} we obtain the estimate
\begin{equation}
\label{Yamaoka Misa}
V(a,\cos\xi_m)\sin\xi_m\leqsim
\sigma(\sin(\pi/k))^{-2\sigma}\leqsim \sigma k^{2\sigma}
\hskip 30pt
\forall m\in\{2,\hdots,l(k)\},
\end{equation}
hence by further majorizing $\sigma\leq1$ we obtain
\begin{equation}
\label{Carl Zeiss}
\nor{\cV_{\text{sum}}}_\infty\leqsim \frac1{k^2}\sum_{m=2}^{l(k)}k^{2\sigma}\leqsim k^{2\sigma-1}.
\end{equation}
From \eqref{gaaaaaaa-gaaaaaaa} we also get the uniform bound $\cos\xi_m\geq\sin(\pi/(2k))$.
This makes us choose $\omega=\sin(\pi/(2k))$, $\sigma_0=1/p$ and let $M=M(\omega,\sigma_0)$ be as in Proposition~\ref{Sekigahara}. Next choose $\nu=\nu(k,p)>0$ such that $u\in(0,\nu)$ implies $\sqrt{1/u^2-1}>M$. Recall that, by Proposition~\ref{Sekigahara},
for such $u$ the inequality \eqref{Damietta} holds. Therefore we may estimate
\begin{equation}
\label{Taka}
\nor{{\cV}_{\text{sum}}}_{L^p(\mu_\varepsilon)}
\leqsim \nor{{\widetilde\cV}_{2}\cdot\chi_{(0,\nu)}}_{L^p(\mu_\varepsilon)}+\nor{{\cV}_{\text{sum}}\cdot\chi_{(\nu,1)}}_{L^p(\mu_\varepsilon)}.
\end{equation}
In order to estimate $\nor{{\widetilde\cV}_{2}\cdot\chi_{(0,\nu)}}_{L^p(\mu_\varepsilon)}$ when $\varepsilon\rightarrow0$,
argue as in the case of $\nor{\cV_{1}}_{L^p(\mu_\varepsilon)}$, see page \pageref{Omodaka Hikokuro}.
Indeed, from \eqref{Damietta}
and \eqref{Yamaoka Misa} 
we deduce
$|{\widetilde\cV}_{2}(u)|
\leqsim
\sigma k^{-2(1-\sigma)}$
for all
$u\in(0,\nu)$.
Therefore
$$
\nor{{\widetilde\cV}_{2}\cdot\chi_{(0,\nu)}}_{L^p(\mu_\varepsilon)}
\leqsim
\frac{\sigma }{k^{2(1-\sigma)}}\nor{\chi_{(0,\nu)}}_{L^p(\mu_\varepsilon)}
\leq
\frac{\sigma }{k^{2(1-\sigma)}}\big(\mu_\varepsilon(0,1)\big)^{1/p}
\leqsim
\frac{\sigma }{k^{2(1-\sigma)}}.
$$

As for the last term in \eqref{Taka}, we use \eqref{Carl Zeiss}, which gives
$$
\nor{{\cV}_{\text{sum}}\cdot\chi_{(\nu,1)}}_{L^p(\mu_\varepsilon)}
\leqsim k^{2\sigma-1}
\bigg(
\int_{[\nu,1]}d\mu_\varepsilon
\bigg)^{1/p}\rightarrow0
\hskip 20pt
\text{ as }\varepsilon\rightarrow0.
$$
This proves that
\[
\limsup_{\varepsilon\searrow0}\nor{{\cV}_{\text{sum}}}_{L^p(\mu_\varepsilon)}
\leqsim \frac{1/p}{k^{2(1-1/p)}}.
\qedhere
\]
\end{proof}

\subsection{Proof of (\ref{West})}

Let $\psi\in(0,\pi/2)$.
By Lagrange's theorem we have, for $b=\sin(\psi/k)$,
\begin{equation}
\label{Shadowboxing}
\aligned
(s^2-2{b}s+1)^{-\sigma}-(s^2+2{b}s+1)^{-\sigma}
&= 4s\sin\frac\psi k\cdot \sigma\xi^{-\sigma-1}
\leqsim \frac{\sigma s}{k\xi^{\sigma+1}}
\endaligned
\end{equation}
for some
$
\xi=\xi(s)\in(s^2-2{b}s+1,s^2+2{b}s+1).
$
Note that
$\xi\geq \left(s-\sin(\psi/ k)\right)^2+\cos^2(\psi/ k).$
Together with \eqref{Scherbakov Symphony no.5} and \eqref{Shadowboxing}, this implies
$$
\aligned
H_\sigma^0\Big(a,\sin\frac\psi k\Big)
& \leqsim
\frac{\sigma}{k}
\int_0^{\sqrt{1/u^2-\cos^2(\psi/k)}-\sin(\psi/k)}
\frac{ds}{\big[\left(s-\sin(\psi/k)\right)^2+\cos^2(\psi/k)\big]^{\sigma+1}}\\
&\leqsim
\frac{\sigma}{k}
\left(\cos\frac\psi k\right)^{-1-2\sigma}\int_{-\tan(\psi/k)}^{\sqrt{1/(u\cos(\psi/k))^2-1}-2\tan(\psi/k)}
\frac{d\tau}{\left(\tau^2+1\right)^{\sigma+1}}.
\endaligned
$$
We simply estimate $\sigma+1>1$, therefore the above integral is majorized by $\int_\mathbb{R}\,d\tau/(\tau^2+1)=\pi$.
Since $k\geq3$, cf. \eqref{kp3}, the cosine term can be estimated as $\sim1$.
Hence we proved
$H_\sigma^0\big(a,\sin(\psi/ k)\big)
\leqsim
\sigma/k$.
This gives
$\cW(u)
\leqsim
\sigma/k^2$,
which proves \eqref{West}.

\section{The second proof of Theorem~\ref{t: Lowest}: Superpositions of dilated Gaussians}
\label{s: Vjeko's}

Our main result in this section is the following strengthened (in the sense of being completely explicit) version of Theorem~\ref{t: Lowest}:

\begin{theorem}
\label{t: LowerExact}
If $p\geq2\geq q$ are such that $1/p+1/q=1$, and $k\in\mathbb{N}$, then
\begin{equation*}
\nor{R^k}_p
\geq
\frac{\Gamma(1/p)\Gamma(1/q+k/2)}{\Gamma(1/q)\Gamma(1/p+k/2)}
\geq
\frac12 k^{1-2/p} (p-1).
\end{equation*}
\end{theorem}

We use that $R$ is a Fourier multiplier with symbol \eqref{eq: Four symbol}.
Consequently, for Schwartz functions $f$ we have
\begin{equation}
\label{eq: R-multiplier}
\widehat{R^k f}(\zeta) = \Big(\frac{\overline{\zeta}}{|\zeta|}\Big)^k \widehat{f}(\zeta)\,.
\end{equation}
Fix $2<p<\infty$ and an odd positive integer $k$.
For each $0<\varepsilon<1$ we are about to define functions $f_{k,p,\varepsilon},g_{p,\varepsilon}\in{L}^2(\mathbb{C})$ such that
\begin{equation}\label{eq:Htranfg}
{R}^k f_{k,p,\varepsilon} = g_{p,\varepsilon}
\end{equation}
and that the ratio $\|g_{p,\varepsilon}\|_{{L}^p(\mathbb{C})}/\|f_{k,p,\varepsilon}\|_{{L}^p(\mathbb{C})}$
is bounded from below by the above quotient of four gamma-functions
in the limit as $\varepsilon\downarrow0$.
We define $g_{p,\varepsilon}$ and $f_{k,p,\varepsilon}$ via their Fourier transforms as
\[ \widehat{g_{p,\varepsilon}}(\zeta) = \int_{\varepsilon}^{1/\varepsilon} e^{-\pi t^2 |\zeta|^2} t^{1-2/p}\,{d}t; \quad \zeta\in\mathbb{C} \]
and
\[ \widehat{f_{k,p,\varepsilon}}(\zeta) = \Big(\frac{\zeta}{|\zeta|}\Big)^k \int_{\varepsilon}^{1/\varepsilon} e^{-\pi t^2 |\zeta|^2} t^{1-2/p}\,{d}t; \quad \zeta\in\mathbb{C}, \]
respectively.
Since
$(\overline{\zeta}/|\zeta|)^k \widehat{f_{k,p,\varepsilon}}(\zeta) = \widehat{g_{p,\varepsilon}}(\zeta)$,
we clearly have \eqref{eq:Htranfg}.

The normalization that we use for the Fourier transform of $f\in{L}^1(\mathbb{R}^d)$ is
\[ \hat{f}(\xi) := \int_{\mathbb{R}^d} f(x) e^{-2\pi i x\cdot\xi} \,{d}A(x); \quad \xi\in\mathbb{R}^d, \]
so that the inverse Fourier transform of $h\in{L}^1(\mathbb{R}^d)$ is given by
\[ \check{h}(x) := \int_{\mathbb{R}^d} h(\xi) e^{2\pi i x\cdot\xi} \,{d}A(\xi); \quad x\in\mathbb{R}^d. \]

The functions $f_{k,p,\varepsilon}$ and $g_{p,\varepsilon}$ essentially approximate constant multiples of $(z/|z|)^k |z|^{-2/p}$ and $|z|^{-2/p}$, respectively.
For instance, pointwise we have
\begin{equation*}
\lim_{\varepsilon\searrow0} g_{p,\varepsilon}(z) = \frac{\Gamma(1/p)}{2\pi^{1/p}} |z|^{-2/p}; \quad z\in\mathbb{C}\setminus\{0\},
\end{equation*}
which will clearly follow from Lemma~\ref{lm:gs} below and the definition of the gamma-function:
\begin{equation}\label{eq:defofgammafn}
\int_{0}^{\infty} x^{s-1} e^{-x} \,{d}x = \Gamma(s)
\end{equation}
for $s>0$.
This observation relates the present approach to the approach from Sections~\ref{s: baba brez maske}--\ref{Zemi ogin}, as the approximate extremizers are similarly behaved.
However, the underlying idea is somewhat different: we wanted to tailor the approximate extremizers to the multiplier symbol.
The present approach is potentially applicable to other Fourier multipliers, even when the kernel or the extremizing sequence cannot be computed explicitly.

The difficulty is now shifted to estimating the ${L}^p$-norms of functions $f_{k,p,\varepsilon}$ and $g_{p,\varepsilon}$.
For this purpose we will find convenient that their frequency representations are superpositions of dilated Gaussians, which will enable us to invert the Fourier transform and compute these functions somewhat explicitly.

We first handle $g_{p,\varepsilon}$, because it is easier.

\begin{lemma}\label{lm:gs}
For $1<p<\infty$ and $z\in\mathbb{C}\setminus\{0\}$ we have
\[ g_{p,\varepsilon}(z) = \frac{1}{2} \pi^{-1/p} |z|^{-2/p} \int_{\pi\varepsilon^2|z|^2}^{\pi\varepsilon^{-2}|z|^2} x^{1/p-1} e^{-x} \,{d}x . \]
\end{lemma}

\begin{proof}
It is a well-known fact (see \cite[Chapter~1,\ Theorem~1.13]{SW71:fa}) that, for any $t>0$, the Fourier transform of
$t^{-2}e^{-\pi t^{-2} |\cdot|^2}$, seen as a function on $\mathbb{C}=\mathbb{R}^2$,
is the function
$e^{-\pi t^{2} |\cdot|^2}$.
Superposing we get the following formula for $g_{p,\varepsilon}$:
\[ g_{p,\varepsilon}(z) = \int_{\varepsilon}^{1/\varepsilon} e^{-\pi t^{-2} |z|^2} \frac{{d}t}{t^{1+2/p}}; \quad z\in\mathbb{C}. \]
A change of variables $x=\pi|z|^2/t^2$ finishes the proof.
\end{proof}

For $0<\varepsilon<1$ define the annulus $A_\varepsilon\subset\mathbb{C}$ as
the region
$$
A_\varepsilon:=\mn{z\in\mathbb{C}}{\varepsilon\leq |z|\leq 1/\varepsilon}.
$$
We will frequently use the following elementary facts:
\begin{equation}
\label{eq: powersLp}
\aligned
\big\|\,|z|^\alpha\,\big\|_{L^p(A_\varepsilon)} & \leq\left(\frac{2\pi}{|2+\alpha p|}\varepsilon^{-|2+\alpha p|}\right)^{1/p} & \text{ if }\alpha\in\mathbb{R}\backslash\{-2/p\}\\
\big\|\,|z|^{-2/p}\,\big\|_{L^p(A_\varepsilon)} & =\left(4\pi\log\frac1\varepsilon\right)^{1/p}\,.
\endaligned
\end{equation}

\begin{notation}
\label{n: Cp}
In Lemma~\ref{lm:gsnorm} below and its proof, $C_p$ will stand for a constant in $\mathbb{R}$
depending on $p$ only, and
whose value may change from one
appearance
to another.
\end{notation}

\begin{lemma}
\label{lm:gsnorm}
For $2<p<\infty$
and some constant $C_p\in\mathbb{R}$
we have
\[ \|g_{p,\varepsilon}\|_{{L}^p(\mathbb{C})} \geq
2^{-(1-2/p)}
\Gamma
(1/p)
\Big(\log\frac{1}{\varepsilon}\Big)^{1/p} +
C_p.\]
\end{lemma}

In fact, we have an equality above (see the proof of Lemma~\ref{lm:fsnorm}), but we do not need the reversed inequality, so we are shortening the proof by a few lines.

\begin{proof}
Choose $z\in\mathbb{C}\backslash\{0\}$.
From Lemma~\ref{lm:gs} and \eqref{eq:defofgammafn} we get
\[ g_{p,\varepsilon}(z) \geq \frac{1}{2} \pi^{-1/p} \Gamma\Big(\frac{1}{p}\Big) |z|^{-2/p} - \frac{1}{2} \pi^{-1/p} |z|^{-2/p} \Big(\int_{0}^{\pi\varepsilon^2|z|^2} + \int_{\pi\varepsilon^{-2}|z|^2}^{\infty}\Big) \,x^{1/p-1} e^{-x} \,{d}x . \]
For the second (i.e., error) term we estimate
\begin{equation*}
\int_{0}^{\pi \varepsilon^2 |z|^2} x^{1/p-1} e^{-x} \,{d}x
\leq \int_{0}^{\pi \varepsilon^2 |z|^2} x^{1/p-1} \,{d}x
\lesssim_{p} (\varepsilon |z|)^{2/p}
\end{equation*}
and
\[ \int_{\pi\varepsilon^{-2}|z|^2}^{\infty} x^{1/p-1} e^{-x} \,{d}x
\leq (\varepsilon/|z|)^{2-2/p} \int_{\pi\varepsilon^{-2}|z|^2}^{\infty} e^{-x} \,{d}x
\leq (\varepsilon/|z|)^{2-2/p} . \]
All this gives, recalling the above convention (Notation \ref{n: Cp}),
\begin{equation}
\label{eq: pointwise lower g}
g_{p,\varepsilon}(z) \geq \frac{1}{2} \pi^{-1/p} \Gamma\Big(\frac{1}{p}\Big) |z|^{-2/p}
+
C_p\varepsilon^{2/p}
+
C_p\varepsilon^{2-2/p}|z|^{-2}.
\end{equation}
It remains to
trivially estimate $\|g_{p,\varepsilon}\|_{{L}^p(\mathbb{C})}\geq\|g_{p,\varepsilon}\|_{{L}^p(A_\varepsilon)}$ and for the latter
apply \eqref{eq: powersLp} to \eqref{eq: pointwise lower g}.
This completes the proof.
\end{proof}

Now we turn to $f_{k,p,\varepsilon}$.

\begin{lemma}
\label{lm:fs}
For $1<p<\infty$, odd $k\in\mathbb{N}$, and $z\in\mathbb{C}\setminus\{0\}$ we have
\[ f_{k,p,\varepsilon}(z)
= i^k \pi^{-1/p-1/2} \Big(\frac{z}{|z|}\Big)^k |z|^{-2/p} \int_{0}^{\pi/2} \Big( \int_{\pi (\sin\vartheta)^2 \varepsilon^2 |z|^2}^{\pi (\sin\vartheta)^2 \varepsilon^{-2} |z|^2} x^{1/p-1/2} e^{-x} \,{d}x \Big) \,\frac{\sin k\vartheta \,{d}\vartheta}{(\sin\vartheta)^{2/p}} . \]
\end{lemma}

\begin{proof}
We need to compute the inverse Fourier transform of $\widehat{f_{k,p,\varepsilon}}$. By Fubini's theorem:
\begin{equation}\label{eq:formforfeps}
f_{k,p,\varepsilon}(z)
= \int_{\varepsilon}^{1/\varepsilon} \bigg(\int_{\mathbb{C}} \Big(\frac{\zeta}{|\zeta|}\Big)^k e^{-\pi t^2|\zeta|^2} e^{2\pi i \textup{Re}(z\overline{\zeta})} \,{d}A(\zeta)\bigg) \,t^{1-2/p}\,{d}t .
\end{equation}
For each fixed $t\in[\varepsilon,\varepsilon^{-1}]$ we pass to the polar coordinates
$z = r e^{i\varphi}$, $\zeta = \rho e^{i\theta}$
in the inner integral (in $\zeta$) and use $2\pi$-periodicity in $\theta$:
\begin{align*}
& \int_{0}^{2\pi} \int_{0}^{\infty} e^{i k \theta} e^{-\pi t^2 \rho^2} e^{2\pi i r \rho \cos(\theta-\varphi)} \rho \,{d}\rho \,{d}\theta
= e^{i k \varphi} \int_{0}^{\infty} \Big( \int_{-\pi}^{\pi} e^{i k \theta} e^{2\pi i r \rho \cos\theta} \,{d}\theta \Big) e^{-\pi t^2 \rho^2} \rho \,{d}\rho.
\end{align*}
Now we split
$e^{i k \theta} = \cos k\theta + i \sin k\theta$
and observe that the integral of an odd function over $[-\pi,\pi]$ vanishes, so that the last display becomes
\begin{equation}\label{eq:oddandeven}
2 e^{i k \varphi} \int_{0}^{\infty} \Big( \int_{0}^{\pi} \cos k\theta \,e^{2\pi i r \rho \cos\theta} \,{d}\theta \Big) e^{-\pi t^2 \rho^2} \rho \,{d}\rho.
\end{equation}
At this point we need to assume that $k$ is odd.
By substituting $\vartheta=\pi/2-\theta$, noting
\[ \cos k\Big(\frac{\pi}{2}-\vartheta\Big) = (-1)^{(k-1)/2}\sin k\vartheta, \]
observing parities of the two terms in
\[ e^{2\pi i r \rho \sin\vartheta} = \cos(2\pi r \rho \sin\vartheta) + i \sin(2\pi r \rho \sin\vartheta), \]
and using Fubini's theorem again, we furthermore transform \eqref{eq:oddandeven} into
\begin{align*}
& (-1)^{(k-1)/2} \,2 e^{i k \varphi} \int_{0}^{\infty} \Big( \int_{-\pi/2}^{\pi/2} \sin k\vartheta \,e^{2\pi i r \rho \sin\vartheta} \,{d}\vartheta \Big) e^{-\pi t^2 \rho^2} \rho \,{d}\rho \\
& = (-1)^{(k-1)/2} \,2 e^{i k \varphi} \int_{0}^{\pi/2} \Big(\int_{0}^{\infty} 2i\rho \,e^{-\pi t^2 \rho^2} \,\sin(2\pi r \rho \sin\vartheta) \,{d}\rho\Big) \sin k\vartheta \,{d}\vartheta . \end{align*}
The inner integral (in $\rho$) is
\[ \int_{\mathbb{R}} \rho \,e^{-\pi t^2 \rho^2} e^{2\pi i r \rho \sin\vartheta} \,{d}\rho \]
and it can be evaluated (for each fixed $0<\vartheta\leq\pi/2$) as
\[ \frac{i r \sin\vartheta}{t^{3}} e^{-\pi t^{-2} r^2 (\sin\vartheta)^2}, \]
simply by using the fact that (for each $t>0$) the Fourier transform of
$\mathbb{R} \ni x \mapsto t^{-3} i x e^{-\pi t^{-2} x^2}$
is the function
$\mathbb{R} \ni \xi \mapsto \xi e^{-\pi t^2 \xi^2}$.
(Just use \cite[Chapter~1,\ Theorem~1.13]{SW71:fa} again and the fact that the Fourier transform interchanges differentiation and multiplication.)
Plugging all this into \eqref{eq:formforfeps} gives
\[ f_{k,p,\varepsilon}(z)
= (-1)^{(k-1)/2} \,2 i e^{i k \varphi} \int_{0}^{\pi/2} \Big( \int_{\varepsilon}^{1/\varepsilon} r \sin\vartheta \,e^{-\pi r^2 (\sin\vartheta)^2 t^{-2}} \frac{{d}t}{t^{2+2/p}} \Big) \sin k\vartheta \,{d}\vartheta. \]
Now we substitute $x=\pi r^2 (\sin\vartheta)^2 t^{-2}$ in the inner integral (the one in $t$), which finally gives
\[ f_{k,p,\varepsilon}(z)
= i^k \pi^{-1/p-1/2} e^{i k \varphi} r^{-2/p} \int_{0}^{\pi/2} \Big( \int_{\pi r^2 (\sin\vartheta)^2 \varepsilon^2}^{\pi r^2 (\sin\vartheta)^2 \varepsilon^{-2}} x^{1/p-1/2} e^{-x} \,{d}x \Big) \,\frac{\sin k\vartheta \,{d}\vartheta}{(\sin\vartheta)^{2/p}}. \]
It remains to recall that $|z|=r$ and $z/|z|=e^{i\varphi}$.
\end{proof}

The final step is estimating the norm $\|f_{k,p,\varepsilon}\|_{{L}^p(\mathbb{C})}$ from above.
For that purpose we need the following technical lemma, which will turn the cancellation in variable $\vartheta$ into an appropriate decay in terms of $k$.

\begin{lemma}
Denote:
\[
I_{k,\alpha} := \int_{0}^{\pi/2} \,\frac{\sin k\vartheta}{(\sin\vartheta)^{2\alpha}} \,{d}\vartheta.
\]
For odd $k\in\mathbb{N}$ and $\alpha\in(0,1)$ we have
\begin{equation}
\label{eq:oscileq}
 I_{k,\alpha}
= \frac{\pi^{1/2}}{2} \cdot \frac{ \Gamma(1-\alpha) \Gamma(\alpha+k/2)}{\Gamma(\alpha+1/2) \Gamma(1-\alpha+k/2)}.
\end{equation}
\end{lemma}

\begin{proof}
For $k=1$
we have
\[
I_{1,\alpha}
= \int_{0}^{\pi/2} (\sin\vartheta)^{1-2\alpha} \,{d}\vartheta
= \frac12 B(1-\alpha,1/2)
= \frac{\Gamma(1-\alpha) \Gamma(1/2)}{2\Gamma(3/2-\alpha)} = \frac{\pi^{1/2}\Gamma(1-\alpha)}{2\Gamma(3/2-\alpha)},
\]
just as we claimed.

Next, suppose that $k\geq3$ is odd and take some $0<\delta<\pi/2$.
Integration by parts gives
\[ \int_{\delta}^{\pi/2} \frac{\sin k\vartheta}{(\sin\vartheta)^{2\alpha}} \,{d}\vartheta
= \frac{\cos k\delta}{k(\sin\delta)^{2\alpha}}
- \frac{2\alpha}{k} \int_{\delta}^{\pi/2} \frac{\cos\vartheta \cos k\vartheta}{(\sin\vartheta)^{1+2\alpha}} \,{d}\vartheta , \]
so that, by multiplying the identity by $k/2$ and applying the addition formula for cosine,
\[ \Big(\frac{k}{2}-\alpha\Big) \int_{\delta}^{\pi/2} \frac{\sin k\vartheta}{(\sin\vartheta)^{2\alpha}} \,{d}\vartheta
= \frac{\cos k\delta}{2(\sin\delta)^{2\alpha}}
- \alpha \int_{\delta}^{\pi/2} \frac{\cos (k-1)\vartheta}{(\sin\vartheta)^{1+2\alpha}} \,{d}\vartheta . \]
Analogously,
\[ \Big(\frac{k-2}{2}+\alpha\Big) \int_{\delta}^{\pi/2} \frac{\sin (k-2)\vartheta}{(\sin\vartheta)^{2\alpha}} \,{d}\vartheta
= \frac{\cos (k-2)\delta}{2(\sin\delta)^{2\alpha}}
- \alpha \int_{\delta}^{\pi/2} \frac{\cos (k-1)\vartheta}{(\sin\vartheta)^{1+2\alpha}} \,{d}\vartheta . \]
Subtracting the last two displays and letting $\delta\to0^+$ gives
\begin{align*}
\Big(\frac{k}{2}-\alpha\Big) I_{k,\alpha} - \Big(\frac{k-2}{2}+\alpha\Big) I_{k-2,\alpha}
& = \lim_{\delta\to0^+} \frac{\cos k\delta-\cos (k-2)\delta}{2(\sin\delta)^{2\alpha}} \\
& = -\lim_{\delta\to0^+} (\sin(k-1)\delta) (\sin\delta)^{1-2\alpha} = 0,
\end{align*}
which transforms into a recurrence relation
\[ I_{k,\alpha} = \frac{k/2 - 1 + \alpha}{k/2 - \alpha} I_{k-2,\alpha} . \]
Desired equality \eqref{eq:oscileq} now follows by the mathematical induction over odd positive integers $k$ using the property $x\Gamma(x)=\Gamma(x+1)$ for $x>0$.
\end{proof}

\begin{lemma}
\label{lm:fsnorm}
For $2<p<\infty$ we have
\[ \|f_{k,p,\varepsilon}\|_{{L}^p(\mathbb{C})} = 2^{2/p} \pi^{-1/2} \Gamma
(1/p+1/2)
I_{k,1/p} \Big(\log\frac{1}{\varepsilon}\Big)^{1/p} +
E_{k,p}(\varepsilon),\]
where
$E_{k,p}$
is a real-valued bounded function of $\varepsilon\in(0,1)$ depending on parameters $k,p$.
\end{lemma}

\begin{proof}
Take $z\in\mathbb{C}\backslash\{0\}$.
By Lemma~\ref{lm:fs} and
\begin{equation}\label{eq:easysine}
\frac{|\sin k\vartheta|}{(\sin\vartheta)^{2/p}}
\leq \frac{|\sin k\vartheta|}{\sin\vartheta} \lesssim_k 1
\end{equation}
we see that $|f_{k,p,\varepsilon}(z)|$ differs from
\begin{align*}
& \pi^{-1/p-1/2} |z|^{-2/p} \bigg| \int_{0}^{\pi/2} \Big( \int_{0}^{\infty} x^{1/p-1/2} e^{-x} \,{d}x \Big) \,\frac{\sin k\vartheta \,{d}\vartheta}{(\sin\vartheta)^{2/p}} \bigg| \\
& = \pi^{-1/p-1/2} \Gamma\Big(\frac{1}{p}+\frac{1}{2}\Big) I_{k,1/p} |z|^{-2/p}
\end{align*}
by at most a constant multiple of
\begin{equation}\label{eq:fepserror}
|z|^{-2/p} \int_{0}^{\pi/2} \bigg( \Big( \int_{0}^{\pi \varepsilon^2 |z|^2} + \int_{\pi (\sin\vartheta)^2 \varepsilon^{-2} |z|^2}^{\infty} \Big) \,x^{1/p-1/2} e^{-x} \,{d}x \bigg) \,{d}\vartheta .
\end{equation}
The integrals in $x$ in this error term are estimated like in the proof of Lemma~\ref{lm:gs} as
\[ \int_{0}^{\pi \varepsilon^2 |z|^2} x^{1/p-1/2} e^{-x} \,{d}x
\leq \int_{0}^{\pi \varepsilon^2 |z|^2} x^{1/p-1/2} \,{d}x
\lesssim (\varepsilon |z|)^{2/p+1} \]
and
\[ \int_{\pi (\sin\vartheta)^2 \varepsilon^{-2} |z|^2}^{\infty} x^{1/p-1/2} e^{-x} \,{d}x
\lesssim_{k,p} \Big(\frac{\varepsilon}{|z|\sin\vartheta}\Big)^{1-2/p} \int_{\pi (\sin\vartheta)^2 \varepsilon^{-2} |z|^2}^{\infty} e^{-x} \,{d}x
\leq \Big(\frac{\varepsilon}{|z|\sin\vartheta}\Big)^{1-2/p} . \]
Therefore,
\eqref{eq:fepserror} is
\[ \lesssim_{k,p} |z|^{-2/p} \big( (\varepsilon |z|)^{2/p+1} + (\varepsilon/|z|)^{1-2/p} \big) . \]
All this gives
\[ |f_{k,p,\varepsilon}(z)| = \pi^{-1/p-1/2} \Gamma\Big(\frac{1}{p}+\frac{1}{2}\Big) I_{k,1/p} |z|^{-2/p} +
C_{k,p}
\varepsilon^{2/p+1}|z|
+
C_{k,p}
\varepsilon^{1-2/p}|z|^{-1}
. \]
(The meaning of $C_{k,p}$ is like in Notation \ref{n: Cp}, just that it may also depend on $k$.)

We first integrate over $A_\varepsilon=\{\varepsilon<|z|<1/\varepsilon\}$. From
\eqref{eq: powersLp}
we get
\[ \Big( \int_{\{\varepsilon\leq |z|\leq 1/\varepsilon\}}|f_{k,p,\varepsilon}(z)|^p \,{d}\textup{A}(z) \Big)^{1/p}
= 2^{2/p} \pi^{-1/2} \Gamma\Big(\frac{1}{p}+\frac{1}{2}\Big) I_{k,1/p} \Big(\log\frac{1}{\varepsilon}\Big)^{1/p} +
C_{k,p}. \]

Now we turn to the region $|z|<\varepsilon$.
This time Lemma~\ref{lm:fs} combined with \eqref{eq:easysine} gives
\begin{align*}
|f_{k,p,\varepsilon}(z)|
& \leq |z|^{-2/p} \Big( \int_{0}^{\pi/2} \frac{|\sin k\vartheta| \,{d}\vartheta}{(\sin\vartheta)^{2/p}} \Big) \Big( \int_{0}^{\pi (|z|/\varepsilon)^2} \! x^{1/p-1/2} e^{-x} \,{d}x \Big) \\
& \lesssim_{k} |z|^{-2/p} \int_{0}^{\pi (|z|/\varepsilon)^2} \! x^{1/p-1/2} \,{d}x
\lesssim_p \varepsilon^{-2/p-1}|z| ,
\end{align*}
so that integrating we get
\[ \Big( \int_{\{|z|<\varepsilon\}}|f_{k,p,\varepsilon}(z)|^p \,{d}\textup{A}(z) \Big)^{1/p}  \lesssim_{k,p} 1. \]

Finally, we turn to the region $|z|>1/\varepsilon$.
Here Lemma~\ref{lm:fs} combined with \eqref{eq:easysine} gives
\begin{align*}
|f_{k,p,\varepsilon}(z)|
& \lesssim_k |z|^{-2/p} \int_{0}^{\pi/2} \Big( \int_{\pi (\sin\vartheta)^2 \varepsilon^2 |z|^2}^{\infty} x^{1/p-1/2} e^{-x} \,{d}x \Big) \,\frac{|\sin k\vartheta| \,{d}\vartheta}{(\sin\vartheta)^{2/p}} \\
& \leq |z|^{-2/p} \int_{0}^{\pi/2} \Big( (\varepsilon|z|)^{2/p-1} \int_{\pi (\sin\vartheta)^2 \varepsilon^2 |z|^2}^{\infty} e^{-x} \,{d}x \Big) \,\frac{|\sin k\vartheta| \,{d}\vartheta}{\sin\vartheta} \\
& \lesssim_k |z|^{-2/p} \int_{0}^{\pi/2} \Big( (\varepsilon|z|)^{2/p} \int_{\pi (\sin\vartheta)^2 \varepsilon^2 |z|^2}^{\infty} e^{-x} \,{d}x \Big) \,{d}\vartheta \\
& = \varepsilon^{2/p} \int_{0}^{\pi/2} e^{-\pi (\sin\vartheta)^2 \varepsilon^2 |z|^2} \,{d}\vartheta ,
\end{align*}
so that Minkowski's integral inequality and polar coordinates give
\begin{align*}
\Big( \int_{\{|z|>1/\varepsilon\}}|f_{k,p,\varepsilon}(z)|^p \,{d}\textup{A}(z) \Big)^{1/p}
& \lesssim_k \varepsilon^{2/p} \int_{0}^{\pi/2} \Big( \int_{1/\varepsilon}^{\infty} e^{-\pi p (\sin\vartheta)^2 \varepsilon^2 r^2} r\,{d}r \Big)^{1/p} \,{d}\vartheta \\
& \leq \varepsilon^{2/p} \int_{0}^{\pi/2} \Big(\frac{1}{2\pi p (\sin\vartheta)^2 \varepsilon^2}\Big)^{1/p} \,{d}\vartheta \\
& \leq \int_{0}^{\pi/2} \!\frac{{d}\vartheta}{(\sin\vartheta)^{2/p}}
\lesssim \int_{0}^{\pi/2} \!\frac{{d}\vartheta}{\vartheta^{2/p}} \lesssim_{p} 1 .
\end{align*}
This completes the proof.
\end{proof}

\subsection{Proof of Theorem~\ref{t: LowerExact}}
For even $k$, the theorem was already established by one of the authors of the present paper in \cite[Theorem 1.1]{D11-2}.

Now we assume that $k$ is odd.
Combining Lemmata~\ref{lm:gsnorm} and \ref{lm:fsnorm} with \eqref{eq:Htranfg} we can write, for a real constant $C_p$ and a bounded real function $E_{k,p}$,
\begin{align*}
\|{R}^k\|_{{L}^p(\mathbb{C})\to{L}^p(\mathbb{C})}
& \geq \frac{\|g_{p,\varepsilon}\|_{{L}^p(\mathbb{C})}}{\|f_{k,p,\varepsilon}\|_{{L}^p(\mathbb{C})}}
\geq \frac{\pi^{1/2}\Gamma(1/p)(\log(1/\varepsilon))^{1/p} +
C_p}{2\Gamma(1/p+1/2)I_{k,1/p}(\log(1/\varepsilon))^{1/p} +
E_{k,p}(\varepsilon)} \\
& = \frac{\pi^{1/2}\Gamma(1/p) +
C_{p}
(\log(1/\varepsilon))^{-1/p}}{2\Gamma(1/p+1/2)I_{k,1/p} +
E_{k,p}(\varepsilon)(\log(1/\varepsilon))^{-1/p}} ,
\end{align*}
so letting $\varepsilon
\searrow0$
and applying the identity \eqref{eq:oscileq} gives the first
inequality.
The second one
is a part of
Proposition~\ref{p: sharp asymptotics} below.
\qed

\subsection{An elementary lower bound}
Taking \cite[Theorem 1.4]{D11-2} as a model, we prove the following bilateral numerical estimate.

\begin{proposition}
\label{p: sharp asymptotics}
For $k\in\mathbb{N}$ and $2\leq p<\infty$ we define
\begin{equation}
\label{eq: gammakp}
\gamma_{k}(p) := \frac{\Gamma(1/p) \Gamma(1/q + k/2)}{\Gamma(1/q) \Gamma(1/p + k/2)}
\end{equation}
where $1/p+1/q=1$.
Then
\begin{equation}\label{eq:gammaest}
\frac12 \leq \frac{\gamma_{k}(p)}{k^{1-2/p} (p-1)} \leq 1
\end{equation}
and the constants $1/2$ and $1$ are both optimal.
\end{proposition}

\begin{proof}
For the proof of \eqref{eq:gammaest} rewrite the ratio as
\[ \frac{\varphi_k(1/q)}{\varphi_k(1/p)}, \]
where
\[ \varphi_k (x) := \frac{\Gamma(x+k/2)}{\Gamma(x+1) k^x}. \]
We claim that $\varphi_k$ is decreasing on $[0,1]$. Once we know that, we just need to observe
\begin{equation*}
\frac{1}{2}
= \frac{\varphi_k(1)}{\varphi_k(0)} \leq \frac{\varphi_k(1/q)}{\varphi_k(1/p)} \leq 1.
\end{equation*}

In terms of the digamma function $\psi$, see Beals and Wong \cite[Section 2.6]{BW10}, we have
\[
\omega_k(x):=\frac{{d}}{{d}x} \log \varphi_k(x)
= \psi(x+k/2) - \psi(x+1) - \log k.
\]
We
need to verify that
$\omega_k(x)\leq 0$ when $k\geq 1$ and $x\in[0,1]$. Applying \cite[(2.6.1)]{BW10} gives
\begin{equation}
\label{eq: omega}
\omega_k(x)
=\sum_{n=0}^\infty\left(\frac1{n+x+1}-\frac1{n+x+k/2}\right)-\log k.
\end{equation}
From \eqref{eq: omega} we see right away that $\omega_k(x)\leq 0$ for $k=1,2$.
Now assume that $k\geq3$. Again from \eqref{eq: omega} we see that
$\omega_k'(x)\leq0$.
Therefore checking $\omega_k(x)\leq0$ is equivalent to checking that $\omega_k(0)\leq0$,
where
\begin{equation}\label{eq: FBB CA 34/P Lasko}
\omega_k(0) = \psi(k/2) - \psi(1) - \log k = \sum_{n=0}^{\infty} \Big(\frac{1}{n+1} - \frac{1}{n+k/2}\Big) - \log k.
\end{equation}
Since, for each integer $k\geq3$,
\[ \omega_{k+2}(0) - \omega_k(0) = \frac{2}{k} - \log\Big(1+\frac{2}{k}\Big) \geq 0, \]
the claim will follow once we verify that the limit $\lim_{k\to\infty}\omega_k(0)$ exists and that it is negative.
The asymptotics of the digamma function for large entries
follows from Binet's integral formula for $\psi$ \cite[p.34]{BW10} and reads
\[ \lim_{z\to\infty} (\psi(z) - \log z) =0, \]
so that \eqref{eq: FBB CA 34/P Lasko} together with \cite[(2.6.1)]{BW10} gives
\[ \lim_{k\to\infty}\omega_k(0) = - \psi(1) -\log 2 = \gamma - \log 2 < 0, \]
where $\gamma$ is the Euler--Mascheroni constant.
\end{proof}

\begin{remark}
With $\gamma_k(p)$ being as in Proposition~\ref{p: sharp asymptotics}, we may check that $\gamma_{k+l}(p)\leq\gamma_{k}(p)\gamma_{l}(p)$. This is necessary if Conjecture~\ref{Konj} is to hold. We leave the proof to the reader, reminding them that in the case of $k,l$ being even, this property was already established in \cite[Proposition 1.3]{D11-2}.
\end{remark}

\section{The third proof of Theorem~\ref{t: Lowest}: Riesz potentials}
\label{sec: Andrea proof}

This section gives yet another proof of the main result by proving its exact variant, Theorem~\ref{t: LowerExact}.
The key idea is that the complex Riesz transform $R$ acts nicely and explicitly on
Riesz potentials,
see Lemma~\ref{l: 1} below.
As opposed to the previous two proofs, this approach does not
handle any oscillatory integrals explicitly.
All the oscillation/cancellation is hidden in the formula for $R^{k}\left[(-\Delta)^{-\alpha}\delta_{0}\right]$ in Lemma~\ref{l: 1}, which is indirectly a consequence of the cancellation of the kernel of $R^{k}$, which occurs because we are taking complex derivatives.

\subsection{Action of $R^k$ on Riesz potentials}
Let $L$ denote the {\it positive} Laplacian on $\mathbb{R}^{2}$. Thus in particular, $Lf=-\Delta f=-(\pd^2_{x}f+\pd^2_{y}f)$ for $f$ in the Schwartz class. Denote by $\delta_{0}$ the Dirac distribution at $0$. For $z=x+iy$ define the complex derivatives by
$$
\partial_{z}=\frac{1}{2}(\partial_{x}-i\partial_{y}),\quad \partial_{\bar z}=\frac{1}{2}(\partial_{x}+i\partial_{y}).
$$
In this section we use the formula \eqref{eq: Delta repres}, thus representing the complex Hilbert transform as the operator
$$
R=(-2i\partial_{z}) L ^{-1/2}=-(\partial_{y} L ^{-1/2}+i\partial_{x} L ^{-1/2}).
$$
The operators $L ^{-\alpha}$ with $\alpha>0$ are called {\it Riesz potentials}, see \cite[Section 1.2]{G14b} for a more in-depth discussion.

Recall that the heat kernel $K_{t}=\exp(-t L )\delta_{0}$ (the convolution kernel of $\exp(-t L )$) on $\mathbb{R}^{2}=\mathbb{C}$  is given by
\begin{equation}\label{eq: Heatkernel}
K_{t}(z)=(4\pi t)^{-1}\exp\left(-|z|^{2}/4t\right),\quad z\in\mathbb{C},\quad t>0.
\end{equation}
For $0<\alpha<1$, the convolution kernel of the Riesz potential is the function
\begin{equation}
\label{eq: KRP}
 L ^{-\alpha}\delta_{0}(z)
 =\frac{1}{\Gamma(\alpha)}\int^{\infty}_{0}t^{\alpha-1}K_{t}(z)\wrt t=\frac{1}{4^{\alpha}\pi}\frac{\Gamma\left(1-\alpha\right)}{\Gamma\left(\alpha\right)}|z|^{2(\alpha-1)}.
\end{equation}
We now calculate the convolution kernels $R^{k} L ^{-\alpha}\delta_{0}$ of the operators $R^{k} L ^{-\alpha}$, $k\geq 0$, $\alpha>0$.
Since $\partial_{z}\bar z=0$, for every $k\in\mathbb{N}$ we have
\begin{equation}
\label{eq: 1}
\partial^{k}_{z}\exp\left(-|z|^{2}/4t\right)=(-1)^{k}\left(\bar z/4t\right)^{k}\exp\left(-|z|^{2}/4t\right).
\end{equation}
Apart from some technicalities, the previous formula allows one to calculate the convolution kernel of $R^{k} L ^{-\alpha}$ in the same way as the 
functions $L ^{-\gamma}\delta_{0}$, $\gamma>0$, are calculated (see, for example \cite[p.\,10]{G14b}).

\medskip
Although, by \eqref{eq: KRP}, the functions $L^{-\alpha}\delta_{0}$ do not belong to $L^p(\C)$ for any $1<p<\infty$, we still manage to define their image under $R^k$ as follows. Fix $\alpha\in(0,1)$, and indices $1<r_{1}<1/(1-\alpha)$ and $r_{2}>1/(1-\alpha)$. By \eqref{eq: KRP}, the
function
$L^{-\alpha}\delta_{0}=L^{-\alpha}\delta_{0}
\cdot 
\chi_{D}+L^{-\alpha}\delta_{0}
\cdot  
 \chi_{\C\setminus D}$ belongs to
$$
X:=L^{r_{1}}(\C)+L^{r_{2}}(\C).
$$
Therefore $R^{k}\left[L^{-\alpha}\delta_{0}\right]$ is well defined, because $R^{k}$ is bounded on $L^{p}(\C)$, for all $p>1$.

\begin{lemma}\label{l: 1}
For every $k\in\mathbb{N}$, $0<\alpha<1$ and $z\in\mathbb{C}\backslash\{0\}$ we have
$$
R^{k} \left[L ^{-\alpha}\delta_{0}\right](z)
=
\frac{(-i)^{k}}{4^{\alpha}\pi}\frac{\Gamma\left(k/2+1-\alpha\right)}{\Gamma\left(k/2+\alpha\right)}\left(\frac{\bar z}{\mod{z}}\right)^{k}\mod{z}^{2(\alpha-1)}.
$$
In particular, for $k=0$, we have the usual formula \eqref{eq: KRP}.
\end{lemma}

\begin{proof}
For $\beta,\epsilon>0$ define the operators
$$
T^{\beta}_{\epsilon}:=L^{\beta}(\epsilon I+L)^{-\beta}=\left[L(\epsilon I+L)^{-1}\right]^{\beta},
$$
and the {\it local complex Riesz transforms}
\begin{equation}
\label{eq: HilbLoc}
R_{\epsilon}:=R\circ T^{1/2}_{\epsilon}=(-2i\partial_{z})(\epsilon I+ L )^{-1/2},\quad \epsilon>0.
\end{equation}
For every $p>1$, the operators $T^{\beta}_{\epsilon}$ converge to the identity in the strong operator topology of $L^{p}(\C)$, as $\epsilon\downarrow 0$. Therefore,
\begin{equation}
\label{eq: Rtrunc}
\lim_{\epsilon\searrow 0}R^{k}_{\epsilon}f=\lim_{\epsilon\searrow 0}R^{k}T^{k/2}_{\epsilon}f=R^{k}f\quad \textrm{in}\ \ L^{p}(\C),\quad \forall f\in L^{p}(\mathbb{C}),\quad \forall p>1.
\end{equation}

Now consider the convolution kernels of the {\it Bessel potentials} $(\epsilon I+L)^{-\alpha}$:
\begin{equation}
\label{eq: Bessel}
(\epsilon I+L)^{-\alpha}\delta_{0}(z)
 =\frac{1}{\Gamma(\alpha)}\int^{\infty}_{0}t^{\alpha-1}K_{t}(z)e^{-\epsilon t}\wrt t,\quad z\in\C,\quad \epsilon>0.
\end{equation}
The functions $(\epsilon I+L)^{-\alpha}\delta_{0}$ converge to $L^{-\alpha}\delta_{0}$ pointwise in $\C\setminus\{0\}$ as $\epsilon\downarrow 0$, and are uniformly dominated by $L^{-\alpha}\delta_{0}\in X$. It thus follows from Lebesgue's dominated convergence theorem that
$$
\lim_{\epsilon\downarrow 0}(\epsilon I +L)^{-\alpha}\delta_{0}=L^{-\alpha}\delta_{0}\quad \textrm{in}\ \ X.
$$
Therefore, by \eqref{eq: Rtrunc},
\begin{equation}
\label{eq: KRPbis}
R^{k}\left[L^{-\alpha}\delta_{0}\right]=\lim_{\epsilon\downarrow 0}R^{k}_{\epsilon}\left[(\epsilon I+L)^{-\alpha}\delta_{0}\right]\quad \textrm{in}\ \ X.
\end{equation}
It follows from the identity
\begin{equation}\label{eq: K norm}
\norm{K_{t}}{{}l{}}=t^{1/l-1}\norm{K_{1}}{{}l{}},\quad 1\leq l\leq \infty,
\end{equation}
that $(\epsilon I+L)^{-\alpha}\delta_{0}\in L^{p}(\C)$ whenever $1\leq p<1/(1-\alpha)$.

Fix $p_{0}\in (1,1/(1-\alpha))$ and consider an approximation of the identity $\psi_{\kappa}(z)=\kappa^{-2}\psi(z/\kappa)$, where $\psi\in C^{\infty}_{c}(\C)$ and $\int_{\C}\psi=1$. Since $R^{k}_{\epsilon}$ is bounded on $L^{p_{0}}(\C)$, we have
\begin{align*}
& R^{k}_{\epsilon}\left[(\epsilon I+L)^{-\alpha}\delta_{0}\right]=\lim_{\kappa\downarrow 0}R^{k}_{\epsilon}\left[(\epsilon I+L)^{-\alpha}\delta_{0}*\psi_{\kappa}\right] \\
& =\lim_{\kappa\downarrow 0}R^{k}_{\epsilon}(\epsilon I+L)^{-\alpha}\psi_{\kappa}=\lim_{\kappa\downarrow 0}\left[R^{k}_{\epsilon}(\epsilon I+L)^{-\alpha}\delta_{0}\right]*\psi_{\kappa}
\end{align*}
from which we deduce
\begin{equation}\label{eq: bn1}
R^{k}_{\epsilon}\left[(\epsilon I+L)^{-\alpha}\delta_{0}\right]=R^{k}_{\epsilon}(\epsilon I+L)^{-\alpha}\delta_{0}=\left[(-2i\partial_{z})^{k}(\epsilon I+L)^{-\alpha-k/2}\right]\delta_{0}\quad \textrm{in}\ \ \cS^{\prime}(\C).
\end{equation}
By combining \eqref{eq: bn1}, \eqref{eq: Bessel} and \eqref{eq: 1}, we obtain
\begin{align*}
& i^{k}R^{k}_{\epsilon} \left[(\epsilon I+L)^{-\alpha}\delta_{0}\right](z) \\
&=\frac{1}{\Gamma\left(k/2+\alpha\right)}(2\partial_{z})^{k}\int^{\infty}_{0}t^{k/2+\alpha-1}K_{t}(z)e^{-\epsilon t}\wrt t\\
&=\frac{1}{\Gamma\left(k/2+\alpha\right)}\int^{\infty}_{0}t^{k/2+\alpha-1}(2\partial_{z})^{k}K_{t}(z)e^{-\epsilon t}\wrt t \\
&=\frac{(-1)^{k}}{4\pi}\frac{1}{\Gamma\left(k/2+\alpha\right)}\left(\frac{\bar z}{2}\right)^{k}\int^{\infty}_{0}t^{-k/2+\alpha-1}e^{-|z|^{2}/4t}e^{-\epsilon t}\frac{\wrt t}{t} \\
&=\frac{(-1)^{k}}{4\pi}\frac{1}{\Gamma\left(k/2+\alpha\right)}\left(\frac{\bar z}{|z|}\right)^{k}2^{-k}|z|^{2(\alpha-1)}\int^{\infty}_{0}t^{-k/2+\alpha-1}e^{-1/4t}e^{-\epsilon t|z|^{2}}\frac{\wrt t}{t}.
\end{align*}
As $\epsilon \downarrow 0$ the integral in the last line converges to
$$
\int^{\infty}_{0}t^{-\beta}e^{-1/4t}\frac{\wrt t}{t}=4^{\beta}\Gamma(\beta),\quad \beta=k/2+1-\alpha,
$$
for all $z\in\C$. The lemma now follows from \eqref{eq: KRPbis}.
\end{proof}

\begin{remark}
When $\alpha\geq 1$ the function $L^{-\alpha}\delta_{0}$ is not in $L^{p}(\C\setminus D)$ for any $p>1$, and we cannot define $R^{k}L^{-\alpha}\delta_{0}$ as we did above in the case $0<\alpha<1$. However, for $0<\alpha<
k/2+1$, the very same argument that we used in the proof of Lemma~\ref{l: 1} shows that
$$
\lim_{\epsilon\downarrow 0}R^{k}_{\epsilon} \left[(\epsilon I+L)^{-\alpha}\delta_{0}\right](z)
=
\frac{(-i)^{k}}{4^{\alpha}\pi}\frac{\Gamma\left(k/2+1-\alpha\right)}{\Gamma\left(k/2+\alpha\right)}\left(\frac{\bar z}{\mod{z}}\right)^{k}\mod{z}^{2(\alpha-1)},\quad \forall z\in\C\setminus\{0\},
$$
and we can take the right-hand side of the identity above as the definition of $R^{k}[L^{-\alpha}\delta_{0}]$ in this range of $\alpha$'s.
\end{remark}

\subsection{Relation to the numbers $\gamma_{k}(p)$}
Let the numbers $\gamma_{k}(p)$ be defined as in \eqref{eq: gammakp}.
\begin{lemma}\label{l: 2}
 Let $p>1$ and $q=p/(p-1)$. For every $\epsilon>0$ define $\alpha_{\epsilon}=1/p+\epsilon$. For every $r>0$ set $B_{r}=B_{\mathbb{C}}(0,r)$. We have
 $$
\gamma_{k}(p) = \lim_{\epsilon\searrow 0}\frac{\norm{R^{k} L ^{-\alpha_{\epsilon}}\delta_{0}}{L^{q}(B_{r})}}{\norm{ L ^{-\alpha_{\epsilon}}\delta_{0}}{L^{q}(B_{r})}}
 $$
 and the quotient inside the limit does not depend on $r>0$.
\end{lemma}
\begin{proof}
The result directly follows from Lemma~\ref{l: 1}. Indeed, it gives
$$
|R^{k} L ^{-\alpha_\epsilon
}\delta_{0}(z)|
=
\frac{\Gamma\left(1/p+\epsilon\right)\Gamma\left(k/2+1/q-\epsilon\right)}{\Gamma\left(1/q-\epsilon\right)\Gamma\left(k/2+1/p+\epsilon\right)} L^{-\alpha_\epsilon}\delta_{0}(z)\,.
$$
Thus, because of \eqref{eq: KRP}, it remains to verify that $|\cdot|^{2(\alpha_\epsilon-1)}\in L^q(B_r)$. A rapid calculation shows that, whenever $r>0$ and $q\vartheta+2>0$,
\begin{equation}
\label{eq: XcX}
\norm{|\cdot|^{\vartheta}}{L^{q}(B_{r})}=\left(\frac{2\pi}{q\vartheta+2}\right)^{1/q}r^{\vartheta+2/q}<\infty\,.
\end{equation}
Applying this for $\vartheta=2(\alpha_\epsilon-1)=2(\epsilon-1/q)$ finishes the proof.
\end{proof}

\subsection{Truncated kernels of Riesz potentials}
For estimating from below the $q$-norm of $R^{k}$ by means of Lemma~\ref{l: 2} the problems are:
\begin{enumerate}[(i)]
\item
inside the limit we need to integrate
over all $\mathbb{C}$ and not just on the ball $B_{r}$, yet
\item
the functions $ L ^{-\alpha_{\epsilon}}\delta_{0}$ do not even lie in $L^{q}(\mathbb{C}\setminus B_{r})$.
\end{enumerate}
We fix the two above problems by means of truncations in the time variable: this is the ``kernel version'' of the truncation in the Fourier domain, used in Section~\ref{s: Vjeko's}.

In the calculations that will follow, the radius $r$ does not play any role (finally, we can choose any $r>0$), but we decided to leave it so as to clarify the calculations.
\medskip

The truncated kernel of $L^{-\alpha}$ (for $\alpha\in (0,1)$) and its complement are defined by
\begin{align}
\label{eq: TRP}
&g^{r^{2}}_{\alpha}(z)=\frac{1}{\Gamma(\alpha)}\int^{r^{2}}_{0}t^{\alpha-1}K_{t}(z)\wrt t,\\
&g^{\alpha}_{r^{2}}(z)=\frac{1}{\Gamma(\alpha)}\int^{\infty}_{r^{2}}t^{\alpha-1}K_{t}(z)\wrt t,\nonumber
\end{align}
 where $K_{t}$ is the heat kernel given by \eqref{eq: Heatkernel}. We clearly have
 \begin{equation}\label{eq: identity}
 g^{r^{2}}_{\alpha}(z)+g^{\alpha}_{r^{2}}(z)= L ^{-\alpha}\delta_{0}(z),\quad \forall z\in\mathbb{C}\setminus\{0\}.
 \end{equation}

\begin{lemma}
For every $\alpha\in (0,1)$ and $r>0$ we have
\begin{equation}\label{eq: ubTRP}
g^{r^{2}}_{\alpha}(z)\leq 2^{1-\alpha}\exp\left(-|z|^{2}/8r^{2}\right) L ^{-\alpha}\delta_{0}(z).
\end{equation}
In particular, for $r>0$, $\alpha\in (1/p,1)$, $p>2$ and $q=p/(p-1)$, we have
$g^{r^{2}}_{\alpha}\in L^{q}(\mathbb{C})$.
Moreover, writing $\alpha_{\epsilon}=1/p+\epsilon$, $0<\epsilon<1/q$, we have the estimate
\begin{equation}\label{eq: updom}
\norm{g^{r^{2}}_{\alpha_{\epsilon}}}{L^{q}(\mathbb{C})}\leq (1+C_{0}\sqrt{\epsilon})\norm{ L ^{-\alpha_{\epsilon}}\delta_{0}}{L^{q}(B_{r})},
\end{equation}
where $C_{0}>0$ is a universal constant that does not depend on $p$, $r$ and $\epsilon$.
\end{lemma}
\begin{proof}
The inequality \eqref{eq: ubTRP} follows from the definitions of $g^{r^{2}}_{\alpha}$, the subordination formula \eqref{eq: KRP} for $ L ^{\alpha}$, the explicit form \eqref{eq: Heatkernel} of the heat kernel, the estimate
$$
\exp\left(-|z|^{2}/4t\right)\leq \exp\left(-|z|^{2}/8r^{2}\right)\exp\left(-|z|^{2}/8t\right),\quad \forall 0<t<r^{2},
$$
and the change of variable $2t=s$ in the integral.

The combination of \eqref{eq: KRP} and \eqref{eq: XcX} shows
that $ L ^{-\alpha}\delta_{0}\in L^{q}_{\rm loc}(\mathbb{C})$ whenever $1/p<\alpha<1$, while the exponential factor in \eqref{eq: ubTRP} gives the integrability at infinity of $g^{r^{2}}_{\alpha}$, for all $\alpha\in (0,1)$.

\medskip
We now prove \eqref{eq: updom}. By the definition of $g^{r^{2}}_{\alpha_{\epsilon}}$ we have
$$
g^{r^{2}}_{\alpha_{\epsilon}}(z)\leq L ^{-\alpha_{\epsilon}}\delta_{0}(z), \quad \forall z\in\mathbb{C}\setminus\{0\}.
$$
Therefore, it suffices to show that there exists $C_{0}>0$ for which
$$
\norm{g^{r^{2}}_{\alpha_{\epsilon}}}{L^{q}(\mathbb{C}\setminus B_{r})}\leq C_{0}\sqrt{\epsilon}\norm{ L ^{-\alpha_{\epsilon}}\delta_{0}}{L^{q}(B_{r})}.
$$
In order to prove the above inequality, observe that by combining the right-hand side of \eqref{eq: KRP} with \eqref{eq: ubTRP} we obtain
$$
\aligned
\norm{g^{r^{2}}_{\alpha_{\epsilon}}}{L^{q}(\mathbb{C}\setminus B_{r})}
& \leq
2^{2-1/p-3\alpha_\epsilon}\pi^{-1/p}
 \frac{\Gamma(1-\alpha_{\epsilon})}{\Gamma(\alpha_{\epsilon})}
  \left(\int^{\infty}_{r}\rho^{2+2q(\alpha_{\epsilon}-1)}e^{-q\rho^{2}/8r^{2}}\frac{\wrt \rho}{\rho}\right)^{1/q}\\
 & \leq C_{1}\frac{\Gamma(1/q-\epsilon)}{\Gamma(1/p+\epsilon)}\,r^{2\epsilon}.
\endaligned
$$
On the other hand,
by \eqref{eq: KRP}  and \eqref{eq: XcX}, applied with $\vartheta=2(\alpha_\epsilon-1)=2(\epsilon-1/q)$, we have
$$
\norm{ L ^{-\alpha_{\epsilon}}\delta_{0}}{L^{q}(B_{r})}
=
\frac{1}{4^{\alpha_{\epsilon}}\pi}\frac{\Gamma(1/q-\epsilon)}{\Gamma(1/p+\epsilon)}\left(\frac{\pi}{q\epsilon}\right)^{1/q}r^{2\epsilon}\geq \epsilon^{-1/2}C_{2}\frac{\Gamma(1/q-\epsilon)}{\Gamma(1/p+\epsilon)}\, r^{2\epsilon}.
$$
Hence we can take $C_{0}=C_{2}/C_{1}$.
\end{proof}
\begin{lemma}\label{l: A2}
Let $r>0$, $p>2$, $q=p/(p-1)$ and $\alpha\in (0,1)$. Then there exists
a universal constant $C_3>0$
such that
$$
\norm{R^{k}g^{\alpha}_{r^{2}}}{L^{q}(B_{r})}\leq \frac{
C_3}{\Gamma(\alpha)}\,r^{2(\alpha-1)+2/q}.
$$
\end{lemma}
\begin{proof}
By \eqref{eq: K norm} and \eqref{eq: TRP}, we have
$g^{\alpha}_{r^{2}}\in L^{{}l{}}(\mathbb{C})$
and the (vector-valued improper Riemann)
integral converges in $L^{{}l{}}(\mathbb{C})$, whenever ${}l{}>(1-\alpha)^{-1}$. Since $R^{k}$ is bounded in $L^{{}l{}}(\mathbb{C})$, $1<{}l{}<\infty$, it follows that
\begin{equation}\label{eq: trivial}
R^{k}g^{\alpha}_{r^{2}}=\frac{1}{\Gamma(\alpha)}\int^{\infty}_{r^{2}}t^{\alpha-1}R^{k}K_{t}\wrt t.
\end{equation}
Recall the notation \eqref{eq: HilbLoc} for the local complex Riesz transform.
Since $e^{-sL}K_t=K_s*K_t=K_{s+t}$ for $t,s>0$, we have
$$
(\delta I+ L )^{-k/2}K_{t}
=\frac{1}{\Gamma(k/2)}\int^{\infty}_{0}e^{-s\delta}s^{k/2-1}e^{-s L }K_{t}\wrt s
=\frac{1}{\Gamma(k/2)}\int^{\infty}_{0}e^{-s\delta}s^{k/2-1}K_{s+t}\wrt s.
$$
By \eqref{eq: 1} we have
$$
(2\partial_{z})^{k}K_{s+t}(z)=(-1)^{k}\left(\frac{\bar z}{2}\right)^{k}\frac{\exp\left(-|z|^{2}/4(s+t)\right)}{4\pi (s+t)^{k+1}},
$$
so that, by \eqref{eq: Rtrunc},
$$
R^{k}K_{t}(z)=\lim_{\delta\searrow 0}R^{k}_{\delta}K_{t}(z)=\left(\frac{\bar z}{2}\right)^{k}\frac{i^{k}}{4\pi\Gamma(k/2)}\int^{\infty}_{0}\frac{s^{k/2-1}}{(t+s)^{k+1}}\exp\left(-|z|^{2}/4(s+t)\right)\wrt s,\quad \forall z\in\mathbb{C}.
$$
Consequently, by \eqref{eq: trivial}, we arrive at the identity
$$
R^{k}g^{\alpha}_{r^{2}}(z)=\frac{\left(
i{\bar z}/2
\right)^{k}}{4\pi\Gamma(\alpha)\Gamma(k/2)}\int^{\infty}_{r^{2}}t^{\alpha-1}\int^{\infty}_{0}\frac{s^{k/2-1}}{(t+s)^{k+1}}\exp\left(-|z|^{2}/4(s+t)\right)\wrt s\wrt t.
$$
Hence,
$$
\aligned
\mod{R^{k}(g^{\alpha}_{r^{2}})(z)}&\leq\frac{\left(
\mod{z}/2
\right)^{k}}{4\pi\Gamma(\alpha)\Gamma(k/2)}\int^{\infty}_{r^{2}}t^{(\alpha-1-k/2)-1}\wrt t\int^{\infty}_{0}\frac{s^{k/2-1}}{(1+s)^{k+1}}\wrt s\\
&=\frac{
\mod{z}^{k} }{2^k4\pi\Gamma(\alpha)\Gamma(k/2)}
\cdot\frac{r^{2(\alpha-1)-k}}{k/2+1-\alpha}
\cdot\frac{\Gamma(k/2)\Gamma(k/2+1)}{\Gamma(k+1)}\\
&=\frac{\mod{z}^{k} }{4^{k+1}\sqrt{\pi}\Gamma(\alpha)}
\cdot\frac{r^{2(\alpha-1)-k}}{k/2+1-\alpha}
\cdot\frac{1}{\Gamma((k+1)/2)}\,.
\endaligned
$$
For the last identity we used the duplication formula for the $\Gamma$ function \cite[Appendix A.8]{G14}.
By recalling \eqref{eq: XcX} with $\vartheta=k$, the lemma now follows with
\begin{equation*}
C_3:=\sup_{k\in\mathbb{N}\atop q\in(1,2), \alpha\in(0,1)}
\frac{\pi^{1/q-1/2}}{2^{2k+2-1/q}
(k/2+1-\alpha)
(qk+2)^{1/q}\Gamma((k+1)/2)}<\infty
\,.
\qedhere
\end{equation*}
\end{proof}

\subsection{Proof of Theorem~\ref{t: Lowest}}
As we have already said, this proof also gives
the
lower bound claimed in Theorem~\ref{t: LowerExact}.
However, via duality discussed in Section~\ref{Clear}, here we choose to control the quantity $\|R^{k}\|_q$ for $1<q<2$, so that we are in fact proving
\begin{equation}\label{eq:addedqnorm}
\nor{R^k}_q \geq \frac{\Gamma(1/p)\Gamma(1/q+k/2)}{\Gamma(1/q)\Gamma(1/p+k/2)}.
\end{equation}

Fix $k\in\mathbb{N}$, exponents $1<q<2<p<\infty$ such that $1/p+1/q=1$ and a number $\alpha=\alpha_{\epsilon}=1/p+\epsilon$ with $0<\epsilon<1/q$. Then,
$$
\frac{\norm{R^{k}g^{1}_{\alpha_{\epsilon}}}{L^{q}(\mathbb{C})}}{\norm{g^{1}_{\alpha_{\epsilon}}}{L^{q}(\mathbb{C})}}\geq \frac{\norm{R^{k}g^{1}_{\alpha_{\epsilon}}}{L^{q}(B_{1})}}{\norm{g^{1}_{\alpha_{\epsilon}}}{L^{q}(\mathbb{C})}}.
$$
By \eqref{eq: identity} and Lemma~\ref{l: A2} (applied with $r=1$) we have
$$
\norm{R^{k}g^{1}_{\alpha_{\epsilon}}}{L^{q}(B_{1})}
\geq \norm{R^{k} L ^{-\alpha_{\epsilon}}\delta_{0}}{L^{q}(B_{1})}-\norm{R^{k}g^{\alpha_{\epsilon}}_{1}}{L^{q}(B_{1})} \\
\geq \norm{R^{k} L ^{-\alpha_{\epsilon}}\delta_{0}}{L^{q}(B_{1})}-\frac{
C_3}{\Gamma(1/p)}.
$$
Hence, by \eqref{eq: updom} (applied with $r=1$),
$$
\frac{\norm{R^{k}g^{1}_{\alpha_{\epsilon}}}{L^{q}(\mathbb{C})}}{\norm{g^{1}_{\alpha_{\epsilon}}}{L^{q}(\mathbb{C})}}\geq\frac{\norm{R^{k} L ^{-\alpha_{\epsilon}}\delta_{0}}{L^{q}(B_{1})}}{(1+C_{0}\sqrt{\epsilon})\norm{ L ^{-\alpha_{\epsilon}}\delta_{0}}{L^{q}(B_{1})}}-\frac{
C_3/\Gamma(1/p)
}{(1+C_{0}\sqrt{\epsilon})\norm{ L ^{-\alpha_{\epsilon}}\delta_{0}}{L^{q}(B_{1})}}.
$$
It follows from the right-hand side of \eqref{eq: KRP} and \eqref{eq: XcX} that
$\norm{ L ^{-\alpha_{\epsilon}}\delta_{0}}{L^{q}(B_{1})}\rightarrow\infty$ as $\epsilon\searrow 0$.
Therefore, by Lemma~\ref{l: 2},
$$
\norm{R^{k}}{q} \geq \limsup_{\epsilon\searrow 0}\frac{\norm{R^{k}g^{1}_{\alpha_{\epsilon}}}{L^{q}(\mathbb{C})}}{\norm{g^{1}_{\alpha_{\epsilon}}}{L^{q}(\mathbb{C})}}\geq \lim_{\epsilon\searrow 0}\frac{\norm{R^{k} L ^{-\alpha_{\epsilon}}\delta_{0}}{L^{q}(B_{1})}}{(1+C_{0}\sqrt{\epsilon})\norm{ L ^{-\alpha_{\epsilon}}\delta_{0}}{L^{q}(B_{1})}}=\gamma_{k}(p).
$$
This concludes the proof of the first inequality claimed in Theorem~\ref{t: LowerExact}.
A concrete bound from Theorem~\ref{t: Lowest} then follows from Proposition~\ref{p: sharp asymptotics}.

\begin{remark}
Let us remark that the above proof gives a slightly stronger variant of Theorem~\ref{t: LowerExact}. Recall that a function $f\colon\mathbb{C}\to\mathbb{C}$ is \emph{radial} if $f(z)=f(|z|)$ for every $z\in\mathbb{C}$. Denote by $L^p_{\textup{rad}}(\C)$ the linear space of all radial functions that are also in $L^p(\C)$.
The heat kernel \eqref{eq: Heatkernel} and all functions constructed from it throughout Section~\ref{sec: Andrea proof} are clearly radial.
Thus, we were considering approximate extremizers among the radial functions only.
Also recall that we did not use the assumption that $k$ is odd anywhere in Section~\ref{sec: Andrea proof}.
That way we have shown a variant of estimate \eqref{eq:addedqnorm} that, after interchanging $p$ and $q$, takes form
\begin{equation*}
\sup_{f\in L^p_{\textup{rad}}(\C),\, \nor{f}_p=1} \nor{R^k f}_p \geq \frac{\Gamma(1/q)\Gamma(1/p+k/2)}{\Gamma(1/p)\Gamma(1/q+k/2)}
\end{equation*}
for $k\in\mathbb{N}$ and $1<p<2$.
\end{remark}

\section{Proofs of Theorems~\ref{t: main}, \ref{v-o-d-a} and \ref{thm: LinftyBMObound}}
\label{Appendix}

Here we demonstrate the remaining results from the Introduction: \eqref{31}, Theorem~\ref{v-o-d-a} and Theorem~\ref{thm: LinftyBMObound}. Arguments needed for the first two already appeared in the unpublished manuscript \cite{D20}. We nevertheless include them here, for the sake of completeness since they naturally complement and motivate other results in this paper.

\subsection{Proof of (\ref{31})}
In this section we prove that \eqref{31} holds for all $k\in\mathbb{N}$, in particular, for the {\it odd} ones. The proof is
basically a repetition of \cite[Section 5]{DPV06}. We present it here for the readers' convenience.

Our plan is to single out constants in the weak (1,1) and strong (2,2) inequalities for ${R}^k$ and then interpolate.
A theorem adequate for our purpose
was proven independently by Christ and Rubio de Francia \cite{CRdF88} and Hofmann \cite{H88}.
The formulation in \cite{H88} is explicit about the behaviour of the estimates, as specified in Theorem~\ref{I'm going upstairs}.
Note that it is valid for kernels far more general (``rougher'') than ours, since no smoothness condition is assumed.

\begin{theorem}[\cite{H88}]
\label{I'm going upstairs}
Suppose $\Omega\in L^q(S^1)$ for some $q>1$ and $\int_{S^1}\Omega=0$. For any $\varepsilon>0$ define the operator $T_\varepsilon$ associated with $\Omega$ and $\varepsilon$ by
\begin{equation}
\label{vrijeme}
T_\varepsilon f(z)=\int_{\{|\zeta|>\varepsilon\}} f(z-\zeta){\Omega(\zeta/|\zeta|)\over|\zeta|^2}\,dA(\zeta).
\end{equation}
Then, for any $\alpha,\varepsilon>0$ and a Schwartz function $f$,
$$
m\{|T_\varepsilon f|>\alpha\}\,\leqsim\,\frac{\nor{\Omega}_q}{\alpha}\,\nor{f}_1.
$$
\end{theorem}

Let us start proving \eqref{31}.
By \eqref{eq: R-multiplier},
each ${R}^k$ is an isometry on $L^2$, so the case $p=2$ is settled.
Now fix $k\in\mathbb{N}$ and consider the case when $1<p<2$.

Let us address the weak (1,1) inequality.
Fix  also $\alpha,\varepsilon>0$ and $f\in L^1$.
Let ${R}_{\varepsilon}^k$ be the ``$\varepsilon$-truncated'' version of ${R}^k$, in the sense of \eqref{vrijeme}.
Denote $M_\varepsilon=\{|{R}_{\varepsilon}^k f|>\alpha\}$ and $M=\{|{R}^k f|>\alpha\}$.

Recall that $\Omega_k$ was introduced in \eqref{eq: Opel Kadett}.
A well-known result on the almost everywhere convergence of homogeneous singular integrals (see, for example, \cite[Corollary 7.11]{MS13}) implies that ${R}^k f=\lim_{\varepsilon\rightarrow 0}{R}_{\varepsilon}^k f$ almost everywhere. This leads to
$\chi_M\leqslant \liminf_{\varepsilon\rightarrow 0}\chi_{M_\varepsilon}$. Consequently, by Fatou's lemma,
$$
m(M)=\int_\mathbb{C}\chi_M\,dm
\leq\int_\mathbb{C}\liminf_{\varepsilon\rightarrow 0}\chi_{M_\varepsilon}\,dm
\leq\liminf_{\varepsilon\rightarrow 0}\int_\mathbb{C}\chi_{M_\varepsilon}\,dm
\leqsim\,\frac{\nor{\Omega_k}_{2}}{\alpha}\,\nor{f}_1
.
$$
In the last inequality we applied Theorem~\ref{I'm going upstairs}, in particular the fact that the estimates there are independent of $\varepsilon$.
So we proved
\begin{equation}
\label{shake it}
m\left\{|{R}^k f|>\alpha\right\}\leqsim \,\frac{k}{\alpha}\,\nor{f}_1.
\end{equation}

Now everything is set for interpolation.
We actually do it as in \cite[Exercise 1.3.2]{G14}.
Choose $p\in (1,2)$ and $r\in (1,p)$.
Let $N_{1,k}$ be the weak (1,1) constant for ${R}^k$.
By the Marcinkiewicz interpolation theorem, applied to $1<r<2$,
\begin{equation}
\label{marc}
\nor{{R}^k}_{r}
\leqslant
2r^{1/r}\bigg({1\over r-1}+{1\over 2-r}\bigg)^{1/r}
N_{1,k}^{\frac{2-r}r}\,.
\end{equation}
The Riesz-Thorin interpolation theorem (i.e., log-convexity of $L^p$ norms) we apply to $r<p<2$ and obtain
\begin{equation}
\label{rt}
\nor{{R}^k}_p\leqslant\nor{{R}^k}_{r}
^{\frac r{2-r}\cdot\frac{2-p}p}.
\end{equation}
By merging \eqref{marc} and \eqref{rt}, applied with
$r=(p+1)/2$, and using that $(p-1)^{1/p}\sim p-1$ for $p\in(1,2)$, we finally arrive at
\begin{equation}
\label{schwach}
\nor{{R}^k}_p\leqsim {N_{1,k}^{
2/p
-1} \over p-1}\,.
\end{equation}
Recall that, according to \eqref{shake it}, we have
\begin{equation}
\label{When was Jesus Born}
N_{1,k}\leqsim k.
\end{equation}
This proves
 \eqref{31}
 for $p\in(1,2)$.
When $p>2$ use
duality (Section \ref{Clear}).

\subsection{Proof of Theorem~\ref{v-o-d-a}.}
Fix $p\in(1,2)$ and combine \eqref{Lowest} with \eqref{schwach}. This gives the estimate $N_{1,k}\geqsim k$.
The proof is finished once we recall \eqref{When was Jesus Born}.

\subsection{Proof of Theorem~\ref{thm: LinftyBMObound}}
For the proof of the upper bound
\[ \nor{R^k}_{L^\infty\rightarrow BMO} \lesssim k \log(k+1) \]
we can rather concentrate on showing the Hardy space bound
\begin{equation}\label{eq: H1L1bound}
\nor{R^k}_{H^1\to L^1} \lesssim k \log(k+1)
\end{equation}
for each $k\in\mathbb{N}$ and then use duality as in Section~\ref{Clear}.
Here it is convenient to regard $BMO$ as the dual of the atomic Hardy space $H^1$.
Straightforward estimation of the kernel gives
\[ | \Omega_k(x-y) - \Omega_k(x) | \lesssim \min\Big\{ \frac{k}{|x|^2}, \frac{k^2 |y|}{|x|^3} \Big\}, \]
$0\neq |x|\geq |y|$, which easily implies the H\"{o}rmander-type condition
\begin{align*}
\int_{|x|\geq 2|y|} | \Omega_k(x-y) - \Omega_k(x) | \,dx
& \lesssim \int_{\{2|y|\leq |x|\leq (k+1)|y|\}} \frac{k}{|x|^2} \,dx + \int_{\{|x|\geq (k+1)|y|\}} \frac{k^2 |y|}{|x|^3} \,dx \\
& \lesssim k \log(k+1)
\end{align*}
for every $y\in\mathbb{C}$.
Now \eqref{eq: H1L1bound} follows directly from \cite[Theorem~2.4.1]{G14b}.

Almost any choice of a function is sufficient to prove the lower bound.
Assume that $k\in\mathbb{N}$ is odd. We use the fact $R^k f_k = g$, where $g(z)=e^{-\pi|z|^2}$ and $f_k$ is given via its Fourier transform as
\[ \widehat{f_k}(\zeta) = \Big(\frac{\zeta}{|\zeta|}\Big)^k e^{-\pi|\zeta|^2}. \]
The very same computation used in the proof of Lemma~\ref{lm:fs}, just without integration in $t$ and simply fixing $t=1$, leads to the formula
\[ f_k(z) = 2 i^k \Big(\frac{z}{|z|}\Big)^k |z| \int_0^{\pi/2} e^{-\pi |z|^2 (\sin\vartheta)^2} \sin\vartheta \sin k\vartheta \,d\vartheta. \]
Integrating by parts we estimate
\begin{align*}
|f_k(z)| & \leq \frac{2|z|}{k} \int_{0}^{\pi/2} e^{-\pi |z|^2 (\sin\vartheta)^2} \big(1 + 2\pi|z|^2(\sin\vartheta)^2\big) \cos\vartheta \,d\vartheta \\
& = \frac{2}{k} \int_0^{|z|} e^{-\pi s^2} (1+2\pi s^2) \,ds \lesssim \frac{1}{k} .
\end{align*}
Thus
\[ \nor{R^k}_{L^\infty\rightarrow BMO} \geq \frac{\nor{R^k f_k}_{BMO(\mathbb{C})}}{\nor{f_k}_{L^\infty(\mathbb{C})}} = \frac{\nor{g}_{BMO(\mathbb{C})}}{\nor{f_k}_{L^\infty(\mathbb{C})}} \gtrsim k. \]
Even powers of $R$ are now estimated simply as
\[ \nor{R^{k+1}}_{L^\infty\rightarrow BMO} \geq \frac{\nor{R^{k+1} f_k}_{BMO(\mathbb{C})}}{\nor{f_k}_{L^\infty(\mathbb{C})}} = \frac{\nor{R g}_{BMO(\mathbb{C})}}{\nor{f_k}_{L^\infty(\mathbb{C})}} \gtrsim k+1. \]

\section*{Acknowledgements}

A. Carbonaro was partially supported by the ``National Group for Mathematical Analysis, Probability and their Applications'' (GNAMPA-INdAM).

O. Dragi\v{c}evi\'c was partially supported by the Slovenian Research Agency, ARRS (research grant J1-1690 and research program P1-0291).

V. Kova\v{c} was supported in part by the Croatian Science Foundation under the project UIP-2017-05-4129 (MUNHANAP).

The authors would like to thank the anonymous referee for their attentive reading of the text and several remarks that improved its presentation.

\bibliography{ComplexHilbert_arXiv3}{}
    \bibliographystyle{alpha}

\end{document}